\documentclass[reqno,twoside,11pt]{amsart}
\usepackage{amsmath,amsfonts,amssymb}
\usepackage{dsfont,mathrsfs,calrsfs,yfonts,eucal}
\usepackage{fullpage,verbatim,color}
\usepackage{hyperref}

\newtheorem{Theorem}{Theorem}[section]
\newtheorem{Definition}[Theorem]{Definition}
\newtheorem{Proposition}[Theorem]{Proposition}
\newtheorem{Lemma}[Theorem]{Lemma}
\newtheorem{Corollary}[Theorem]{Corollary}
\newtheorem{Remark}[Theorem]{Remark}
\newtheorem{Example}[Theorem]{Example}
\newtheorem{Hypothesis}{Hypothesis}

\numberwithin{equation}{section}

\def\R{\mathbb R}
\def\N{\mathbb N}

\def\E{\mathbb E}
\def\P{\mathbb P}

\def\ds{\displaystyle}

\begin{document}

\title{Absolutely continuous  solutions for continuity equations in Hilbert spaces}

\author{Giuseppe Da Prato}
\address{Scuola Normale Superiore, Piazza dei Cavalieri 7, 56126 Pisa, Italy}

\author{Franco Flandoli}
\address{Scuola Normale Superiore, Piazza dei Cavalieri 7, 56126 Pisa, Italy}

\author{Michael R\"ockner}
\address{Bielefeld University, Universit\"atsstrasse 25, 33615 Bielefeld, Germany and \newline Academy of Mathematics and Systems Science, CAS, Beijing, China}

\keywords{Continuity equations, non Gaussian measures, rank condition}
\subjclass[2010]{35F05, 58D20, 60H07}

\begin{abstract}  
 We prove existence of solutions to continuity  equations in a separable Hilbert space. We look for solutions which are absolutely continuous with respect to a reference measure $\gamma$ which is  Fomin--differentiable with exponentially integrable partial logarithmic derivatives. We describe a class of examples to which our result applies and for which we can prove also uniqueness. Finally, we consider the case where $\gamma$ is the invariant measure of a reaction--diffusion equation and prove uniqueness of solutions in this case. We exploit  that  the gradient operator $D_x$  is closable with respect to $L^p(H,\gamma)$ and a recent formula for the commutator $D_xP_t-P_tD_x$ where $P_t$ is the transition semigroup corresponding  to the reaction--diffusion equation, \cite{DaDe14}.  We stress that $P_t$    is not necessarily symmetric in this case. 
This uniqueness result is an extension to such  $\gamma$ of that in \cite{DaFlRoe14} where $\gamma$ was the Gaussian invariant measure of a suitable Ornstein--Uhlenbeck process.\bigskip

\noindent
{\sc R\'esum\'e}.
On d\'emontre l'existence d'une  solution de quelques \'equations  de  continuit\'e  dans un espace de Hilbert  s\'eparable. On s'interesse aux  solutions absolument
 continues par rapport \`a une mesure de reference $\gamma$ que l'on suppose d\'erivable au sens de Fomin et ayant
  les   deriv\'ees partielles
logarithmiques  exponentiellement int\'egrables. 
On d\'ecrit une classe d'exemples a qui nos r\'esultats s' appliquent et dont on peut aussi montrer l'unicit\'e. 
Finalment on consid\`ere le cas o\`u   $\gamma$ est la mesure invariante 
d'une \'equation de  r\'eaction--diffusion  dont  l'on prouve 
l'unicit\'e des solutions.
On utilise le fait que le gradient   $D_x$  est fermable dans $L^p(H,\gamma)$  et aussi une  r\'ecente formule   pour le commutateur  $D_xP_t-P_tD_x$,   $P_t$  \'etant le s\'emigroupe de transitions  qui
corr\'espond \`a l'\'equation de  r\'eaction--diffusion consider\'ee \cite{DaDe14}.  On souligne que  dans ce cas  $P_t$   n'est pas n\'ecessairement sym\'etrique. 
Ce r\'esultat d'unicit\'e est une extension de celui obtenu 
dans \cite{DaFlRoe14} ou  $\gamma$ \'et\'e la mesure invariante  Gaussienne d'un processus  de Ornstein--Uhlenbeck  appropri\'e. 
\end{abstract}

\maketitle

\section{Introduction}
We are given     a separable Hilbert space  $H$ (norm $|\cdot|_H$, inner product  $\langle \cdot,\cdot  \rangle$),    a Borel vector field $F:[0,T]\times H\to H$ and a Borel probability measure $\zeta$ on $H$. We are concerned with the  following continuity equation,
 \begin{equation}
\label{e1.1}
\int_0^T\int_H\left[D_tu(t,x)+\langle D_xu(t,x),F(t,x)  \rangle\right]\,\nu_t(dx)\,dt=-\int_Hu(0,x)\,\zeta(dx),\quad \forall\;u\in \mathcal F C^1_{b,T},
\end{equation} 
 where the unknown $\mbox{\boldmath}\,\nu=(\nu_t)_{t\in[0,T]}$ is a  probability kernel such that $\nu_0=\zeta$. Moreover, $D_x$ represents the gradient operator  and $\mathcal F C^1_{b,T}$ is defined as follows:  let $\mathcal F C^k_{b}$ and $\mathcal F C^k_{0}$, for $k\in\N\cup\{\infty\}$, denote the set of all functions $f:H\to\R$ of the form
 $$
 f(x)=\widetilde f(\langle  h_1,x \rangle,\cdots, \langle h_N,x   \rangle),\quad x\in H,
 $$
 where $N\in\N$, $\widetilde f\in C^k_b(\R^N)$, $ C^k_0(\R^N)$ respectively (i.e. $ \widetilde f$ has compact support)  and $h_1,\cdots,  h_N\in Y,$ where $Y$ is a dense linear subspace of $H$ to be specified later. 
 Then $\mathcal F C^k_{b,T}$ is defined to be the $\R$--linear span of all functions $u:[0,T]\times H\to\R$ of the form
 $$
 u(t,x)=g(t)f(x),\quad (t,x)\in[0,T]\times H,
 $$
 where $g\in C^1([0,T];\R)$ with $g(T)=0$ 
 and $f\in \mathcal F C^k_{b}$. Correspondingly, let $\mathcal V\mathcal F C^k_{b,T}$ be the set of all maps $G:[0,T]\times H\to H$  of the form
 \begin{equation}
\label{e1.3m}
G(t,x)=\sum_{i=1}^Nu_i(t,x)h_i,\quad (t,x)\in[0,T]\times H,
 \end{equation}
 where $N\in\N$, $u_1,\cdots,u_N\in \mathcal F C^k_{b,T}$ and $h_1,\cdots,  h_N\in Y$. 
 Clearly, $\mathcal F C^\infty_{b,T}$ is dense in $L^p([0,T]\times H,\nu)$ for all finite Borel measures $\nu$ on $[0,T]\times H$ and all $p\in[1,\infty)$.  $\mathcal V\mathcal F C^k_{b}$ denotes the set of all $G$ as in \eqref{e1.3m} with $u_i\in \mathcal F C^k_{b,T}$  replaced by $u_i\in \mathcal F C^k_{b}.$ 
   Of course, all these spaces $\mathcal F C^k_{b}$, $\mathcal F C^k_{0}$, $\mathcal F C^k_{b,T}$, $\mathcal V\mathcal F C^k_{b}$, 
 $\mathcal V\mathcal F C^k_{b,T}$ depend on $Y$. But since $\gamma$ in Hypothesis \ref{h1} below  will be fixed  and hence the corresponding $Y$ defined there will be fixed we do not express this dependence in the notation.

 It is well known that problem \eqref{e1.1}  in general  admits several solutions even when $H$ is finite dimensional. So, it is  natural to  look for well posedness of \eqref{e1.1} within   the special class of  measures $(\nu_t)_{t\in[0,T]}$ which are absolutely continuous with respect to a given {\em reference     measure} $\gamma$. In this case, denoting by $\rho(t,\cdot)$ the density of $\nu_t$  with respect to $\gamma$,
$$
\nu_t(dx)=\rho(t,x)\gamma(dx),\quad
t\in[0,T],
$$
equation  \eqref{e1.1} becomes
  \begin{equation}
  \begin{array}{l}
\label{e1.2}
\ds\int_0^T\int_H\left[D_tu(t,x)+\langle D_xu(t,x),F(t,x)  \rangle\right]\,\rho(t,x)\,\gamma(dx)\,dt\\
\\
\ds =-\int_Hu(0,x)\,\rho_0(x)\gamma(dx),\quad \forall\;u\in\mathcal F C^1_{b,T}.
\end{array}
\end{equation} 
Here $\rho_0:=\rho(0,\cdot)$ is given and $\rho(t,\cdot),\;t\in[0,T],$ is the unknown.

In this paper we prove existence and uniqueness results for solutions to \eqref{e1.2}.  
 
Our   basic assumption on $\gamma$ is the following
\begin{Hypothesis}
\label{h1}
   $\gamma$ is a nonnegative measure on $(H,\mathcal B(H))$  with $\gamma(H)<\infty$ such that there exists a dense linear subspace $Y\subset H$ having the following properties:
   
   For all $h\in Y$ there exists $\beta_h:H\to\R$ Borel measurable such that for some $c_h>0$
   $$
   \int_H e^{c_h|\beta_h|}\,d\gamma<\infty
   $$
   and
   $$
   \int_H\partial_h u\,d\gamma=-\int_Hu\beta_h\,d\gamma,
   $$
  where  $\partial_hu$ denotes the partial derivative of $u$ in the direction $h$.
 \end{Hypothesis}

Assume from now on that $\gamma$ satisfies Hypothesis \ref{h1}.
\begin{Remark}
\label{r1.1}
It is well known that the operator $D_x=$ Fr\'echet--derivative with domain $\mathcal FC^1_b$ is closable in $L^p(H,\gamma)$ for all $p\in[1,\infty)$, see e.g. \cite{AlRoe90}. Its closure will  again be denoted by  $D_x$ and its domain will be denoted by $W^{1,p}(H,\gamma)$.
\end{Remark}

Let $D_x^*:dom(D_x^*)\subset L^2(H,\gamma;H)\to  L^2(H,\gamma)$ denote the adjoint of $D_x$.

\begin{Lemma}
\label{l1.2}
$\mathcal V\mathcal FC^1_b\subset dom(D_x^*)$ and for $G\in \mathcal V\mathcal FC^1_b$, $G=\sum_{i=1}^Nu_ih_i$ we have
$$
D_x^*G=-\sum_{i=1}^N(\partial_{h_i}u_i+\beta_{h_i}u_i).
$$
\end{Lemma} 
\begin{proof}
For $v\in \mathcal FC^1_b$ we have
$$
\begin{array}{l}
\ds\int_H \langle D_xv,G   \rangle_H\,d\gamma=\sum_{i=1}^N\int_H\partial_{h_i}v\,u_i
\,d\gamma\\
\\
\ds=\sum_{i=1}^N\int_H\partial_{h_i}(v\,u_i)
\,d\gamma-\sum_{i=1}^N\int_Hv\,\partial_{h_i}u_i
\,d\gamma\\
\\
\ds=-\int_H v\,\sum_{i=1}^N(\partial_{h_i}u_i+\beta_{h_i}u_i)\,d\gamma.
\end{array} 
$$
\qed
\end{proof}
We stress that if $H$ is infinite dimensional, $\beta_h$ is typically not bounded and not continuous. Here are some examples. For $G$ as in Lemma \ref{l1.2}, below we sometimes use the notation
$$
\mbox{\rm div}\;G:=\sum_{i=1}^N\partial_{h_i}u_i.
$$
\begin{Example}
\label{ex1.3}
\em (i) Let $Q$ be a symmetric positive definite operator of trace class  on $H$ and $\gamma:=N(0,Q),$ i.e. the centered Gaussian measure on $H$ with covariance operator $Q$. Assume that ker $Q=\{0\}$  and let $Y$ be the linear span of all eigenvectors of $Q$. Then Hypothesis \ref{h1} is fulfilled with this $Y$ and for $h\in Y$, $h=a_1h_1+\cdots +a_Nh_N$ with $Qh_i=\lambda_i^{-1}h_i,$ we have
$$
\beta_h(x)=-\sum_{i=1}^Na_i\lambda_i \langle h_i,x  \rangle_H,\quad x\in H.
$$
This, in particular,  covers the case studied in \cite{DaFlRoe14}, where only uniqueness of solutions to \eqref{e1.2} was studied.

(ii) Let $H:=L^2((0,1),d\xi)$ and $A:=\Delta$ with zero boundary conditions.

We recall that $N(0,\tfrac12(-A)^{-1})((C([0,1];\R))=1$. Define for $p\in(2,\infty)$ and $\alpha\in[0,\infty)$
$$
\gamma(dx):=\frac1{Z}\;e^{-\frac\alpha{p}\int_0^1|x(\xi)|^pd\xi}\;N(0,\tfrac12\,(-A)^{-1})(dx),
$$
where
$$
Z:=\int_H e^{-\frac\alpha{p}\int_0^1|x(\xi)|^pd\xi}\;N(0,\tfrac12\,(-A)^{-1})(dx).
$$
Then with $Y$ as in (i) for $Q=\frac12\,(-A)^{-1}$ we find for $h=a_1h_1+\cdots +a_Nh_N$ as in (i)
\begin{equation}
\label{e1.3a}
\beta_h(x)=-\sum_{i=1}^Na_i\left(\lambda_i \langle h_i,x  \rangle_H+\alpha\int_0^1h_i(\xi)\,|x(\xi)|^{p-2}\,x(\xi)\,d\xi\right) \quad\mbox{\rm for $N(0,\tfrac12\,(-A)^{-1})$--a.e. $x\in H$}
\end{equation}
 and obviously also  the exponential integrability condition holds   in Hypothesis \ref{h1}.\medskip
 
 \noindent(iii)  Let $H$ and $A$  be as in (ii) and let  $\gamma$ be  the invariant measure of the solution to 
 \begin{equation}
\label{e3.1bis}
\left\{\begin{array}{lll}
dX(t)=[AX(t)+p(X(t))]dt+BdW(t),\\
\\
X(0)=x,\quad x\in H,
\end{array}\right.
\end{equation}
where   $p$ is a decreasing polynomial   
of odd degree equal to $N>1$,  $B\in L(H)$ with a {\em bounded inverse} and  $W$  is an $H$--valued cylindrical Wiener process  on a filtered
probability space
$(\Omega,\mathcal F, (\mathcal F_t)_{t>0},\P)$ (see \cite{DaDe17}). 
Then it was proved in \cite[Proposition 3.5]{DaDe17} that Hypothesis \ref{h1} holds with $Y:=D(A)$, where $A$ is as in (ii) above except that each  $\beta_h$ was only proved to be $L^p(L^2(0,1),\gamma)$ for every $p\ge 1$. More precisely, it was proved (see \cite[eq. (3.17)]{DaDe17} ) that for all $h\in D(A)$
 $$
\left( \int_{L^2(0,1)} |\beta_h|^p\,d\gamma \right)^{\frac1p} \le C_p|Ah|,\quad \forall\;p\ge 2,
 $$
 where $C_p$  is the constant of the Burkholder--Davis--Gundy inequality for $p\ge 2$ which (when proved by It\^o's formula) can easily  be  seen to be smaller than $12\,p$ if $p\ge 4$. For the reader's convenience we include a proof in Appendix B below. Hence, because for all $n\in\N$ by Stirling's formula
 $$
 \left(\frac1{n!}12^n\,n^n   \right)^{\frac1n}\le 12n\left( \frac1{\sqrt{2\pi}}\,n^{-n-\frac12} \,e^n \right)^{\frac1n}=12e\left( \frac1{\sqrt{2\pi}}\right)^{\frac1n}e^{-\frac1{2n}\ln n}\to 12e\quad\mbox{\rm as}\;n\to\infty,
 $$
   we have for all $\epsilon\in(0,(12e|Ah|)^{-1}), h\in D(A)\setminus \{0\}$,
 $$
 \int_{L^2(0,1)} e^{\epsilon|\beta_h|}\,d\gamma\le \sum_{n=0}^\infty \frac1{n!}\epsilon^n12^n\,n^n |Ah|^n<\infty.
 $$
 So, for any $c_h\in (0,(12e|Ah|)^{-1})$,  exponential integrability holds for $|\beta_h|$ and Hypothesis \ref{h1} is satisfied.
 
 Define for an orthonormal basis $\{e_i,\,i\in\N\}$ of $H$ consisting of elements in $Y$ and $N\in\N$
$$
H_N:=\mbox{\rm  lin span}\,\{e_1,...,e_N\}
$$
and let  $\Pi_N:H\to E_N$ be the   orthogonal projection onto $E_N:=H_N^\perp$, where $H_N^\perp$ is the orthogonal complement of $H_N$, i.e.
\begin{equation}
\label{e2.9n}
H=H_N\oplus E_N\equiv \R^N\times E_N,
\end{equation}
hence, for $z\in H,$ $z=(x,y)$ with unique $x\in\R^N$, $y\in E_N$.

 Letting $\nu_N:=\gamma\circ \Pi^{-1}_N$ be the image measure on  $(E_N,\mathcal B(E_N))$ of $\gamma$ under $ \Pi_N$. Then we have the following well known disintegration result for $\gamma$:
\begin{Lemma}
\label{l2.2n}
There exists $\Psi_N:\R^N\times E_N\to [0,\infty)$,  $\mathcal B(\R^N\times E_N)$--measurable such that
\begin{equation}
\label{e2.10n}
\gamma(dz)=\gamma(dx\,dy)=\Psi_N^2(x,y)dx\,\nu_N(dy),
\end{equation}
where $dx$ denotes Lebesgue measure on $\R^N$.
Furthermore, for every $y\in E_N$
\begin{equation}
\label{e2.11n}
\Psi_N(\cdot,y)\in H^{1,2}(\R^N,dx),
\end{equation}
i.e. the Sobolev space of order $1$ in $L^2(\R^N,dx)$.
\end{Lemma}
\begin{proof}
See \cite[Proposition 4.1]{AlRoeZh93}.
\qed
\end{proof}
  We have by  Hypothesis \ref{h1} that for all $1\le i\le N$ there exists $c_i\in (0,\infty)$ such that
\begin{equation}
\label{e1.8primo}
\begin{array}{lll}
\ds\infty&>&\ds\int_H e^{c_i|\beta_{e_i}|}\,d\gamma=\int_{E_N}\int_{\R^N}e^{c_i|\beta_{e_i}(x,y)|}\,\Psi_N^2(x,y)\,dx\,\nu_N(dy)\\
\\
&=&\ds
\int_{E_N}\int_{\R^N}\exp\left[c_i\,\left|\frac{\partial }{\partial x_i}\,\Psi_N^2(x,y)/\Psi_N^2(x,y)\right|\right]\,\Psi_N^2(x,y)dx\,\nu_N(dy),
\end{array} 
\end{equation}
where we used that for $1\le i\le N$
\begin{equation}
\label{e2.12n}
\beta_{e_i}(x,y)=\frac{\partial }{\partial x_i}\,\Psi_N^2(x,y)/\Psi_N^2(x,y),\quad (x,y)\in \R^N\times E_N=H,
\end{equation}
which is an immediate consequence of the disintegration \eqref{e2.10n}, and the right hand side of \eqref{e2.12n} is defined to be zero on $\{\Psi_N=0\}$. Hence 
\begin{equation}
\label{e2.13n}
\int_{\R^N}\exp\left[c_i\left|\frac{\partial }{\partial x_i}\,\Psi_N^2(x,y)/\Psi_N^2(x,y)\right|\right]\,\Psi_N^2(x,y)dx<\infty\quad\mbox{\rm  for $\nu_N$-a.e., $y\in E_N$}
\end{equation}

 Define for $M,l\in\N$ and $(x,y)\in \R^N\times E_N(=H)$
$$
\Psi_{N,M,l}(x,y)=\Psi_N(x,y)\quad\mbox{\rm if}\,\,\Psi_N^2(\cdot,y)\,\mbox{\rm is $C^2$, strictly positive and bounded}
$$
and otherwise
\begin{equation}
\label{e2.15n}
\Psi_{N,M}(x,y):=\left( \Psi_N^2(x,y)\wedge M\vee M^{-1}   \right)^{1/2},
\end{equation}
\begin{equation}
\label{e2.16n}
\Psi_{N,M,l}(x,y):=\left( \Psi^2_{N,M}(\cdot,y)*\delta_l   \right)^{1/2}(x),
\end{equation}
where $\delta_l (x)=l^N\eta(lx),$ $x\in\R^N$, $\eta\in C^\infty_0(\R^N)$  with support in the unit ball, $\eta\ge 0$, $\eta(x)=\eta(-x)$, $x\in\R^N,$ and $\int_{\R^N}\eta\,dx=1$). We note that then clearly
$\Psi_{N,M,l}(x,y)\ge M^{-1}$ for all $x\in\R^N$.
Obviously,
  \begin{equation}
\label{e2.15'}
\frac{\partial _{x_i}\,\Psi^2_{N,M,l}(\cdot,y)}{\Psi^2_{N,M,l}(\cdot,y)}
\to \frac{\partial _{x_i}\,\Psi^2_{N,M}(\cdot,y)}{\Psi^2_{N,M}(\cdot,y)}\quad\mbox{\rm in}\, L^1_{loc}(\R^N,dx)\; \mbox{\rm as}\;l\to\infty,\;\forall\;y\in E_N,\;1\le i\le N.
 \end{equation}

\end{Example}
Concerning $F$ in  \eqref{e1.1} we assume for $\gamma$ and   $Y$ given as in Hypothesis \ref{h1}.
\begin{Hypothesis}
\label{h2}
(i) $F:[0,T]\times H\to H$ is Borel measurable and bounded.\medskip

(ii) There exists an orthonormal basis $\{e_n,\,n\in\N\}$ of $H$ consisting of elements in $Y$  such that  for  every $N\in\N$ and  $\nu_N$ a.e. $y\in E_N$
\begin{equation}
\label{e1.14m}
\frac{\partial _{x_i}\,\Psi^2_{N}(\cdot,y)}{\Psi^2_{N}(\cdot,y)}\in  L^1_{loc}(\R^N,dx)
\end{equation}
(Please see the ``Note added in Proof" after the acknowledgement).\medskip

\noindent (iii) There exist $F_j:[0,T]\times H\to H,$  $j\in\N$, such that for some $N_j\in\N$ increasing in $j$,
$$
F_j(t,x)=\sum_{i=1}^{N_j}f_{ij}(t,x)e_i,\quad (t,x)\in[0,T]\times H,
$$
(with $e_i$ as in (ii)), where for $1\le i\le N_j$
$$
f_{ij}(t,x)=\widetilde f_{ij}(t,( \langle x,e_1   \rangle,..., \langle x,e_{N_j}  \rangle )) 
$$
with $\widetilde f_{ij}\in C_b([0,T]\times\R^{N_j};\R)$  and $\widetilde f_{ij}(t,\cdot)\in C^2_b (\R^{N_j};\R)$ for all $t\in [0,T]$ such that all first and all second partial derivatives are in $C([0,T]\times\R^{N_j};\R)$,
$$
\left\{\begin{array}{l}
\ds \lim_{j\to\infty} F_j=F\quad dt\otimes\gamma\mbox{\rm-a.e.}\\\\
\ds\sup_{j\in\N}\,\|F_j\|_\infty<\infty,\\
\\
\ds \exists\,\delta>0\;\mbox{\it such that}\, M:= \sup_{j\in\N}\,C_{F_j}(\delta)<\infty,
\end{array}\right. 
$$
 where $C_{F_j}(\delta):=\int_{E_{N_j}}C_{F_j}(\delta,y)\,\nu_{N_j}(dy)$ and
  $$
C_{F_j}(\delta,y):=\sup_{M,l\in\N}\int_0^T\left(\int_{\R^{N_j}}e^{\delta(D^*_{N_j,M,l}F_j(t,x,y))^+}-1\right)\,\Psi_{N_j,M,l}^2(x,y)\,dx\,dt,
$$
 with
\begin{equation}
\label{e2.18n}
D^*_{N_j,M,l}F_j(t,(x,y)):=-\sum_{i=1}^{N_j} \left(  \partial_{e_i}f_{ij}(t,x)+f_{ij}(t,x)\,\frac{\partial }{\partial x_i}\,\Psi_{N_j,M,l}^2(x,y)/\Psi_{N_j,M,l}^2(x,y)  \right).
\end{equation}

\end{Hypothesis}
\begin{Remark}
\label{r1.5primo}
\em
We shall see in Example \ref{ex2.10} below that Hypothesis \ref{h2}(ii) is trivially fulfilled  in Examples \ref{ex1.3}(i) and (ii). Whether it holds
in  Example \ref{ex1.3}(iii) is an open problem (see Remark \ref{r3.12primo} below) and will be a subject of  further study.  
\end{Remark}

 Here is an abstract condition  which ensures Hypothesis \ref{h2}. Some concrete examples will be given later.  
 
\begin{Proposition}
\label{p1.6}
 Let $\gamma$ be a nonnegative measure satisfying Hypothesis 1;\ let $\Psi
_{N}\left(  x,y\right)  $ be defined by (1.7). Let $\Lambda:H\rightarrow H$ be
a positive selfadjoint Hilbert-Schmidt operator with $\Lambda e_{n}%
=\epsilon_{n}e_{n}$, for a sequence $\left\{  \epsilon_{n}\right\}  $ such
that $\sum_{n=1}^{\infty}\epsilon_{n}^{2}<\infty$. Let $F:\left[  0,T\right]
\times H\rightarrow H$ satisfying the conditions below. Assume:

i) $\Psi_{N}\left(  \cdot,y\right)  $ is of class $C^{2}\left(  \mathbb{R}%
^{N}\right)  $, bounded and strictly positive for all $y\in E_{N}$

ii) $F=\Lambda F_{0}$, where $F_{0}:\left[  0,T\right]  \times H\rightarrow H$
is uniformly continuous and bounded

iii) (divergence bounded from below) for some constant $C\geq0$
\[
\sum_{n=1}^{N}\partial_{e_{n}}\left\langle F\left(  t,x\right)  ,e_{n}%
\right\rangle \geq-C\qquad\text{for every }N\text{ and }x\in H
\]

iv) for some constants $\delta>0$
\[
\int_{H}e^{\delta\sum_{n=1}^{\infty}\epsilon_{n}\left\vert \beta_{e_{n}%
}\left(  x\right)  \right\vert }\nu\left(  dx\right)  <\infty.
\]

 Then Hypothesis 2 is fulfilled.
\end{Proposition}

\begin{proof}
\textbf{Step 1} (definition of $F_{N}$). In the verification of Hypothesis 2
we shall take $N_{j}=j$ hence, for simplicity of notations, we use $N$ in
place of $j$. For every $n,N\in\mathbb{N}$ with $n\leq N$ define
$\widetilde{f}_{n,N}^{0},\widetilde{f}_{n,N}:\left[  0,T\right]
\times\mathbb{R}^{N}\rightarrow\mathbb{R}$ as%
\[
\widetilde{f}_{n,N}^{0}\left(  t,x_{1},...,x_{N}\right)  =\left\langle
F_{0}\left(  t,\sum_{i=1}^{N}x_{i}e_{i}\right)  ,e_{n}\right\rangle
\]%
\[
\widetilde{f}_{n,N}\left(  t,x_{1},...,x_{N}\right)  =\left\langle F\left(
t,\sum_{i=1}^{N}x_{i}e_{i}\right)  ,e_{n}\right\rangle =\epsilon
_{n}\widetilde{f}_{n,N}^{0}\left(  t,x_{1},...,x_{N}\right)  .
\]
For every $N\in\mathbb{N}$, let $\theta^{N}:\mathbb{R}^{N}\rightarrow
\mathbb{R}$ be a smooth probability density with support in the unit ball of
center zero and for every $\delta>0$ set%
\[
\theta_{\delta}^{N}\left(  x\right)  =\delta^{-N}\theta^{N}\left(  \delta
^{-1}x\right)  .
\]
Let $\left(  \delta_{N}\right)  $ be an infinitesimal sequence. Define
$f_{n,N}^{0},f_{n,N}:\left[  0,T\right]  \times\mathbb{R}^{N}\rightarrow
\mathbb{R}$ as%
\[
f_{n,N}^{0}\left(  t,x_{1},...,x_{N}\right)  =\left(  \theta_{\delta_{N}}%
^{N}\ast\widetilde{f}_{n,N}^{0}\left(  t,\cdot\right)  \right)  \left(
x_{1},...,x_{N}\right)  .
\]
\[
f_{n,N}\left(  t,x_{1},...,x_{N}\right)  =\left(  \theta_{\delta_{N}}^{N}%
\ast\widetilde{f}_{n,N}\left(  t,\cdot\right)  \right)  \left(  x_{1}%
,...,x_{N}\right)  =\epsilon_{n}f_{n,N}^{0}\left(  t,x_{1},...,x_{N}\right)  .
\]
Then define%
\[
F_{N}\left(  t,x\right)  =\sum_{n=1}^{N}f_{n,N}\left(  t,\left\langle
x,e_{1}\right\rangle ,...,\left\langle x,e_{N}\right\rangle \right)  e_{n}.
\]
The structure and regularity of $F_{N}\left(  t,x\right)  $ are obviously
satisfied. 

\textbf{Step 2} (convergence of $F_{N}$). We prove here that the sequence of
functions $F_{N}\left(  t,x\right)  $ converges pointwise to $F\left(
t,x\right)  $. Let $\left(  t,x\right)  \in\left[  0,T\right]  \times H$ be
given. From the inequalities%
\begin{align*}
& \left| \sum_{n=1}^{N}f_{n,N}\left(  t,\left\langle x,e_{1}\right\rangle
,...,\left\langle x,e_{N}\right\rangle \right)  e_{n}-\sum_{n=1}^{\infty
}\left\langle F\left(  t,x\right)  ,e_{n}\right\rangle e_{n}\right|
_{H}^{2}\\
& \leq2\sum_{n=1}^{N}\left(  f_{n,N}\left(  t,\left\langle x,e_{1}%
\right\rangle ,...,\left\langle x,e_{N}\right\rangle \right)  -\left\langle
F\left(  t,x\right)  ,e_{n}\right\rangle \right)  ^{2}+2\sum_{n=N+1}^{\infty
}\left\langle F\left(  t,x\right)  ,e_{n}\right\rangle ^{2}\\
& \leq2\sum_{n=1}^{N}\epsilon_{n}^{2}\left(  f_{n,N}^{0}\left(  t,\left\langle
x,e_{1}\right\rangle ,...,\left\langle x,e_{N}\right\rangle \right)
-\left\langle F_{0}\left(  t,x\right)  ,e_{n}\right\rangle \right)
^{2}+2\left\Vert F_{0}\right\Vert _{\infty}^{2}\sum_{n=N+1}^{\infty}%
\epsilon_{n}^{2}%
\end{align*}
and the convergence of $\sum_{n=1}^{\infty}\epsilon_{n}^{2}<\infty$ we see
that it is sufficient to prove%
\[
\lim_{N\rightarrow\infty}\sup_{n\leq N}\left(  f_{n,N}^{0}\left(
t,\left\langle x,e_{1}\right\rangle ,...,\left\langle x,e_{N}\right\rangle
\right)  -\left\langle F_{0}\left(  t,x\right)  ,e_{n}\right\rangle \right)
^{2}=0.
\]
Since (a priori we have to write $\lim\sup$ instead of $\lim$)%
\begin{align*}
& \lim_{N\rightarrow\infty}\sup_{n\leq N}\left(  \left\langle F_{0}\left(
t,\sum_{i=1}^{N}\left\langle x,e_{i}\right\rangle e_{i}\right)  ,e_{n}%
\right\rangle -\left\langle F_{0}\left(  t,x\right)  ,e_{n}\right\rangle
\right)  ^{2}\\
& \leq\lim_{N\rightarrow\infty}\sum_{n=1}^{N}\left\langle F_{0}\left(
t,\sum_{i=1}^{N}\left\langle x,e_{i}\right\rangle e_{i}\right)  -F_{0}\left(
t,x\right)  ,e_{n}\right\rangle ^{2}\\
& \leq\lim_{N\rightarrow\infty}\left|F_{0}\left(  t,\sum_{i=1}%
^{N}\left\langle x,e_{i}\right\rangle e_{i}\right)  -F_{0}\left(  t,x\right)
\right| _{H}^{2}=0
\end{align*}
because of the uniform continuity of $F_{0}$, we see it is sufficient to prove
that
\[
\lim_{N\rightarrow\infty}\sum_{n=1}^{N}\left(  f_{n,N}^{0}\left(
t,\left\langle x,e_{1}\right\rangle ,...,\left\langle x,e_{N}\right\rangle
\right)  -\left\langle F_{0}\left(  t,\sum_{i=1}^{N}\left\langle
x,e_{i}\right\rangle e_{i}\right)  ,e_{n}\right\rangle \right)  ^{2}=0.
\]
Denote $\left\langle F_{0}\left(  t,\sum_{i=1}^{N}\left\langle x,e_{i}%
\right\rangle e_{i}\right)  ,e_{n}\right\rangle $ by $h_{n,N}\left(
t,x\right)  $. We have
\begin{align*}
& \sum_{n=1}^{N}\left\vert f_{n,N}^{0}\left(  t,\left\langle x,e_{1}%
\right\rangle ,...,\left\langle x,e_{N}\right\rangle \right)  -h_{n,N}\left(
t,x\right)  \right\vert ^{2}\\
& =\sum_{n=1}^{N}\left\vert \left(  \theta_{\delta_{N}}^{N}\ast\widetilde{f}%
_{n,N}^{0}\left(  t,\cdot\right)  \right)  \left(  \left\langle x,e_{1}%
\right\rangle ,...,\left\langle x,e_{N}\right\rangle \right)  -h_{n,N}\left(
t,x\right)  \right\vert ^{2}\\
& \leq\int_{\mathbb{R}^{N}}\theta_{\delta_{N}}^{N}\left(  ...,\left\langle
x,e_{j}\right\rangle -x_{j}^{\prime},...\right)  \sum_{n=1}^{N}\left\vert
\left\langle F_{0}\left(  t,\sum_{i=1}^{N}x_{i}^{\prime}e_{i}\right)
,e_{n}\right\rangle -h_{n,N}\left(  t,x\right)  \right\vert ^{2}dx_{1}%
^{\prime}...dx_{N}^{\prime}\\
& \leq\int_{\mathbb{R}^{N}}\theta_{\delta_{N}}^{N}\left(  ...,\left\langle
x,e_{j}\right\rangle -x_{j}^{\prime},...\right)  \left\Vert F_{0}\left(
t,\sum_{i=1}^{N}x_{i}^{\prime}e_{i}\right)  -F_{0}\left(  t,\sum_{i=1}%
^{N}\left\langle x,e_{i}\right\rangle e_{i}\right)  \right\Vert ^{2}%
dx_{1}^{\prime}...dx_{N}^{\prime}.
\end{align*}
Since $\theta^{N}$ has support in the unit ball of center zero, $\theta
_{\delta_{N}}^{N}$ has support in the ball or radius $\delta_{N}$ and center
zero. Denoting by $\eta_{N}$ the numbers (related to modulus of continuity)%
\[
\eta_{N}=\sup_{\left| \sum_{i=1}^{N}x_{i}^{\prime}e_{i}-\sum_{i=1}%
^{N}\left\langle x,e_{i}\right\rangle e_{i}\right|_H \leq\delta_{N}%
}\left| F_{0}\left(  t,\sum_{i=1}^{N}x_{i}^{\prime}e_{i}\right)
-F_{0}\left(  t,\sum_{i=1}^{N}\left\langle x,e_{i}\right\rangle e_{i}\right)
\right|
\]
we have%
\[
\sum_{n=1}^{N}\left\vert f_{n,N}^{0}\left(  t,\left\langle x,e_{1}%
\right\rangle ,...,\left\langle x,e_{N}\right\rangle \right)  -h_{n,N}\left(
t,x\right)  \right\vert ^{2}\leq\eta_{N}^{2}.
\]
Since $\delta_{N}\rightarrow0$ and $F_{0}$ is uniformly continuous, we deduce
$\eta_{N}^{2}\rightarrow0$ and the proof is complete. The proof of the
equi--boundedness of the family $F_{N}\left(  t,x\right)  $ is similar (we only
sketch the main steps):%
\begin{align*}
\left|F_{N}\left(  t,x\right)  \right|_{H}^{2} &  =\sum_{n=1}%
^{N}\left(  f_{n,N}\left(  t,\left\langle x,e_{1}\right\rangle
,...,\left\langle x,e_{N}\right\rangle \right)  \right)  ^{2}\\
&  =\sum_{n=1}^{N}\epsilon_{n}^{2}\left(  f_{n,N}^{0}\left(  t,\left\langle
x,e_{1}\right\rangle ,...,\left\langle x,e_{N}\right\rangle \right)  \right)
^{2}\leq\left\Vert F_{0}\right\Vert _{\infty}^{2}\sum_{n=1}^{\infty}%
\epsilon_{n}^{2}.
\end{align*}

\textbf{Step 3} (exponential bound). Finally, let us check the last
condition of Hypothesis 2. Since $\Psi_{N}\left(  \cdot,y\right)  $ is of
class $C^{2}\left(  \mathbb{R}^{N}\right)  $ and bounded, we can take
$\Psi_{N,M,l}\left(  x,y\right)  =\Psi_{N}\left(  \cdot,y\right)  $. If
$G_{N}\left(  x\right)  =\sum_{n=1}^{N}u_{n}\left(  x\right)  e_{n}$, then,
with the notations used above,
\[
D_{N,M,l}^{\ast}G_{N}\left(  x,y\right)  =-\sum_{n=1}^{N}\left(
\partial_{e_{n}}u_{n}\left(  x\right)  +u_{n}\left(  x\right)  \beta_{e_{n}%
}\left(  x,y\right)  \right)  .
\]
Hence%
\begin{align*}
& D_{N,M,l}^{\ast}F_{N}\left(  t,\left(  x,y\right)  \right)  \\
& =-\sum_{n=1}^{N}\left(  \partial_{e_{n}}f_{n,N}\left(  t,\left\langle
x,e_{1}\right\rangle ,...,\left\langle x,e_{N}\right\rangle \right)
+f_{n,N}\left(  t,\left\langle x,e_{1}\right\rangle ,...,\left\langle
x,e_{N}\right\rangle \right)  \beta_{e_{n}}\left(  x,y\right)  \right)
\end{align*}%
\[
\leq-\left(  \theta_{\delta_{N}}^{N}\ast\sum_{n=1}^{N}\partial_{e_{n}%
}\widetilde{f}_{n,N}\left(  t,\cdot\right)  \right)  \left(  \left\langle
x,e_{1}\right\rangle ,...,\left\langle x,e_{N}\right\rangle \right)
+\sum_{n=1}^{N}\epsilon_{n}\left\vert f_{n,N}^{0}\left(  t,\left\langle
x,e_{1}\right\rangle ,...,\left\langle x,e_{N}\right\rangle \right)
\right\vert \left\vert \beta_{e_{n}}\left(  x,y\right)  \right\vert .
\]
But%
\[
\sum_{n=1}^{N}\partial_{e_{n}}\widetilde{f}_{n,N}\left(  t,x_{1}%
,...,x_{N}\right)  =\sum_{n=1}^{N}\partial_{e_{n}}\left\langle F\left(
t,\sum_{i=1}^{N}x_{i}e_{i}\right)  ,e_{n}\right\rangle \geq-C
\]
hence%
\[
-\left(  \theta_{\delta_{N}}^{N}\ast\sum_{n=1}^{N}\partial_{e_{n}%
}\widetilde{f}_{n,N}\left(  t,\cdot\right)  \right)  \left(  \left\langle
x,e_{1}\right\rangle ,...,\left\langle x,e_{N}\right\rangle \right)  \leq C.
\]
And%
\[
\left\vert f_{n,N}^{0}\left(  t,\left\langle x,e_{1}\right\rangle
,...,\left\langle x,e_{N}\right\rangle \right)  \right\vert \leq\left\vert
\left(  \theta_{\delta_{N}}^{N}\ast\widetilde{f}_{n,N}^{0}\left(
t,\cdot\right)  \right)  \left(  \left\langle x,e_{1}\right\rangle
,...,\left\langle x,e_{N}\right\rangle \right)  \right\vert
\]%
\[
\leq\int_{\mathbb{R}^{N}}\theta_{\delta_{N}}^{N}\left(  ...,\left\langle
x,e_{j}\right\rangle -x_{j}^{\prime},...\right)  \left\vert \left\langle
F_{0}\left(  t,\sum_{i=1}^{N}x_{i}e_{i}\right)  ,e_{n}\right\rangle
\right\vert dx_{1}^{\prime}...dx_{N}^{\prime}\leq\left\Vert F_{0}\right\Vert
_{\infty}.
\]
Summarizing,%
\[
D_{N,M,l}^{\ast}F_{N}\left(  t,\left(  x,y\right)  \right)  \leq C+\left\Vert
F_{0}\right\Vert _{\infty}\sum_{n=1}^{N}\epsilon_{n}\left\vert \beta_{e_{n}%
}\left(  x,y\right)  \right\vert
\]
and thus, finally,%
\begin{align*}
&  \sup_{N\in\mathbb{N}}\int_{E_{N}}\sup_{M,l\in\mathbb{N}}\left(  \int%
_{0}^{T}\int_{\mathbb{R}^{N}}e^{\delta D_{N,M,l}^{\ast}F_{N}\left(  t,\left(
x,y\right)  \right)  }\Psi_{N,M,l}^{2}\left(  x,y\right)  dxdt\right)  \nu
_{N}\left(  dy\right)  dt\\
&  \leq T\int_{H}e^{\delta\left[  C+\left\Vert F_{0}\right\Vert _{\infty}%
\sum_{n=1}^{\infty}\epsilon_{n}\left\vert \beta_{e_{n}}\left(  x\right)
\right\vert \right]  }\nu\left(  dx\right)  <\infty
\end{align*}
for some $\delta>0$.
\qed
\end{proof}

\begin{Definition}
Let $\rho_0\in L^1(H,\gamma))$. A {\em  solution} of the continuity equation \eqref{e1.2}  is a function   $\rho\in L^1(0,T; L^1(H,\gamma)$ such that  $\rho(0,\cdot)=\rho_0$ 
and \eqref{e1.2} is fulfilled.
    \end{Definition}

    If $\rho_0\ln\rho_0\in L^1(H,\gamma)$, in Section 2,  we shall prove existence of a   solution of  \eqref{e1.2}  by introducing the following approximating equation, where $F$ is replaced by $(F_j)$ (fulfilling Hypothesis \ref{h2}) and $\rho_0$ by  $\rho_{j,0}$, where  $(\rho_{j,0})$ is a sequence in $\mathcal F C^1_b$, converging to $\rho_0$ in $L^1(H,\gamma)$:
    \begin{equation}
\label{e1.4.a}
\begin{array}{l}
\ds\int_0^T\int_H\left[D_tu(t,x)+\langle D_xu(t,x),F_j(t,x)  \rangle\right]\,\rho_j(t,x)\,\gamma(dx)\,dt\\
\\
\ds=-\int_Hu(0,x)\,\rho_{j0}(x)\gamma(dx),\quad \forall\;u\in\mathcal F C^1_{b,T},
\end{array}
\end{equation} 
which has a solution $\rho_j$ since $F_j$ is regular. Then we shall show that a subsequence of $(\rho_{j})$ converges weakly to a solution of \eqref{e1.2}.
 In Section 3 we prove uniqueness of solutions to \eqref{e1.2} for a whole class of (non--Gaussian) reference measures $\gamma$ based on an infinite dimensional
analogue of DiPerna--Lions type commutator estimates (see \cite{DiLi89}). 

 We present a whole explicit class of examples to which our results apply, i.e. for which we have both existence and uniqueness of solutions to \eqref{e1.2} (see Example \ref{ex2.10} below).

  To our knowledge,   earliest existence (and uniqueness) results for equation \eqref{e1.2}  concern  the case where $H$ is finite dimensional and the reference   measure is the Lebesgue measure, see the seminal papers    \cite{DiLi89} and \cite{Am04}. If $H$ is infinite dimensional   and $\gamma$ is a Gaussian measure, problem  \eqref{e1.1} has been studied in \cite{AmFi09}, \cite{FaLu10} and \cite{DaFlRoe14}. In \cite{KoRoe14} also   non--Gaussian measures, $\gamma$, e.g. Gibbs measures were   studied. However, only in the case where $F$ does not depend  on $t$.   A very general approach in   metric spaces   has been presented in \cite{AmTr14}, but under the assumption div$_\gamma F$ is bounded.  
  Our assumptions  for getting existence of solutions, however, 
do not require   div$_\gamma F$
 to be bounded and our uniqueness results include cases where the reference measure $\gamma$  is not Gaussian and not even Gibbsian, i.e. the smoothing semigroup $P_\epsilon$ is not symmetric on $L^2(H,\gamma)$.
 
 \medskip

We finish this section with some notations  and preliminaries.  $\mathcal B(H)$ denotes
the set of all Borel  subsets   and $\mathcal P(H)$ the set of all Borel probabilities on $H$. A {\em probability kernel} in $[0,T]$ is a mapping
 $ [0,T]\to\mathcal P(H),\; t\mapsto \mu_t,$
 such that     the mapping $
[0,T]\to \R,\;  t\mapsto \mu_t(I) $
 is measurable for any $I\in\mathcal B(H)$.
$L(H)$ is  the set of all linear bounded operators in $H$,
  $C_b(H)$, $C_b(H;H)$     the space of all real continuous   and bounded mappings   $\varphi\colon H\to \R$ and  $\varphi\colon H\to H$ respectively, endowed with the sup norm 
  $$\|\varphi \|_{\infty}=\sup_{x\in  H}\,|\varphi(x)|,$$
  whereas $C^k_b(H)$, $k>1,$ will denote the  space of all real functions which are continuous and bounded together with their derivatives of order less or equal to $k$.
       $ B_b(H)$ will represent the space of  all  real,  bounded   and Borel mappings on $H$. Moreover,  we shall denote by $\|\cdot\|_p$ the norm in $L^p(H,\gamma)$, $p\in[1,\infty]$. For any $x,y\in H$ we denote either by $\langle x,y\rangle$ or by $x\cdot y$ the scalar product  between $x$ and $y$. Finally,  if $(e_h)$ is an orthonormal basis in $H$ we set $x_h=\langle x,e_h\rangle$ for all $x\in H$ and
$G_h=\langle G,e_h   \rangle,\;h\in\N,$ for all $G\in L^2(H,\nu;H)$. Finally, we state a lemma, needed in what follows, whose straightforward proof is
left to the reader.
\begin{Lemma}
\label{l1.4}
Assume, besides Hypothesis \ref{h1}, that $F\in$ dom $(D_x^*)$ and $\varphi\in C^1_b(H)$. Then $\varphi F\in$ dom $(D_x^*)$ and we have
\begin{equation}
\label{e}
 D_x^*(\varphi F)= \varphi \,D_x^*(F)-\langle D_x\varphi, F\rangle.
 \end{equation}

\end{Lemma}

 \section{The main existence result}

 First we notice  that  if    $F\in \mbox{\rm dom}\,(D^*_x)$   then a regular solution $\rho$ to \eqref{e1.2}  solves the equation
 \begin{equation}
\label{e2.4a}
\left\{\begin{array}{l}
\ds  D_t\rho+\langle F,D_x\rho  \rangle-D_x^*F\,\rho=0,\\
\\
\rho(0,\cdot)=\rho_0,
\end{array}\right. 
\end{equation}
and vice-versa. In fact, since for all $u\in \mathcal FC^1_{b,T}$
 \begin{equation}
\label{e2.1a}
\int_0^TD_tu(t,x)\,\rho(t,x)\,dt=-\int_0^Tu(t,x)\,D_t\rho(t,x)\,dt-u(0,x)\rho(0,x),\quad x\in H,
\end{equation}
and (thanks to Lemma \ref{l1.4})
\begin{equation}
\label{e2.2a}
\begin{array}{l}
\ds \int_H\langle D_xu(t,x),F(t,x)  \rangle\,\rho(t,x)\,\gamma(dx)=\int_H\langle D_xu(t,x),\rho(t,x)F(t,x)  \rangle\, \gamma(dx)\\
\\
\ds=\int_H u(t,x)\,D_x^*(\rho F)(t,x) \, \gamma(dx)=
\int_H u(t,x)\,\rho(t,x)\,D_x^*F(t,x) \, \gamma(dx)\\
\\
\ds 
-\int_H u(t,x)\, \langle D_x\rho(t,x),F(t,x)   \rangle\, \gamma(dx).
\end{array} 
\end{equation} 
Clearly \eqref{e2.1a} and \eqref{e2.2a} imply that \eqref{e1.2} is equivalent to
 \begin{equation}
\label{e2.3a}
\left\{\begin{array}{l}
\ds\int_0^T\int_H u(t,x)\left[- D_t\rho(t,x)+  D^*_xF(t,x)\rho(t,x)
-\langle D_x\rho(t,x),F(t,x)   \rangle\,\right]\gamma(dx)\,dt=0,\\
\\
\rho(0,\cdot)=\rho_0,
\end{array}\right.
\end{equation} 
  for all $u\in \mathcal FC^1_{b,T}.$
By the density   of $\mathcal V \mathcal FC^1_{b,T}$ in $L^2([0,T]\times H,dt\otimes d\gamma)$ we obtain \eqref{e2.4a}.
 \begin{Theorem}
\label{t2.1n}
Assume that Hypotheses \ref{h1} and \ref{h2} hold. Let $\zeta:=\rho_0\cdot\gamma$ be a probability measure  on $(H,\mathcal B(H))$ such that
\begin{equation}
\label{e2.5n}
\int_H\rho_0\,\ln \rho_0\,d\gamma<\infty.
\end{equation}
Then there exists $\rho: [0,T]\times H\to \R_+$, $\mathcal B([0,T]\times H)$--measurable such that $\nu_t(dx)=\rho(t,x)\gamma(dx)$, $t\in[0,T]$,  are probability measures  on $(H,\mathcal B(H))$ such that
\eqref{e1.1} (equivalently \eqref{e1.2}) holds. In addition,
\begin{equation}
\label{e2.6n}
\int_0^T\int_H\rho(t,x)\,\ln \rho(t,x)\,\gamma(dx)\,dt<\infty.
\end{equation}
\end{Theorem}
\begin{proof}
By disintegration we shall reduce the proof to the case $H=\R^N$ and by regularization to Corollary \ref{cA2} in Appendix A.   Let  $\{e_n,\,n\in\N\}$  be the  orthonormal basis  from Hypothesis \ref{h2}(ii)\medskip

\noindent{\bf Case 1}. Suppose first that  $F:[0,T]\times H\to H  $ is as an $F_j$ from   Hypothesis \ref{h2}(iii),
$ \rho_0\in \mathcal FC^1_{0}$, $\rho_0\ge 0$.\medskip

 Hence for some  $N\in\N$ (which we fix below and shall no longer explicitly express in the notation below, i.e. write $\Psi_{N,M,l}$ as $\Psi_{M,l}$, $E$ instead of $E_N$, etc.)
\begin{equation}
\label{e2.7n}
F(t,x)=\sum_{i=1}^Nf_i(t,x)\,e_i,\quad (t,x)\in[0,T]\times H,
\end{equation}
where for $1
\le i\le N$, 
$$
f_i(t,x)=\widetilde{f_i}(t,\langle e_1,x  \rangle,...,\langle e_N,x  \rangle)
$$
and
$$
\rho_0(x)=\widetilde{\rho_0}(\langle e_1,x  \rangle,...,\langle e_N,x  \rangle)
$$
with   $\widetilde{\rho_0}\in C^1_0(\R^N)$ and $\widetilde f_i$ as in Hypothesis \ref{h2}(iii).

Then by Corollary \ref{cA2} applied with $\Psi=\Psi^2_{M,l}(\cdot,y)$, we know that
\begin{equation}
\label{e2.17}
 \rho_{M,l}(t,(x,y)):=\rho_0(\xi(T,T-t,x))\,e^{\int_0^tD^*_{M,l}\,F(T-u,(\xi(T-u,T-t,x),y))\,du},\quad (t,x)\in[0,T]\times \R^N,
\end{equation}
where (see Lemma\ref{l1.2} and \eqref{e2.18n})
\begin{equation}
\label{e2.18nn}
D^*_{M,l}F_j(r,(x,y)):=-\sum_{i=1}^N \left(  \partial_{e_i}f_{ij}(t,x)+f_{ij}(t,x)\,\frac{\partial }{\partial x_i}\,\Psi_{M,l}^2(x,y)/\Psi_{M,l}^2(x,y)  \right),
\end{equation}
$r\in[0,T],\;x\in\R^N$, solves
\begin{equation}
\label{e2.19n}
\left\{\begin{array}{l}
\ds D_t\rho_{M,l}(t,(x,y))+ \langle F(t,x),D_x\rho_{M,l}(t,(x,y))   \rangle-D^*_{M,l}F(t,(x,y))\rho_{M,l}(t,(x,y))=0,\\
\\
\rho_{M,l}(0,(x,y))
=\rho_0(x).
\end{array}\right. 
\end{equation}
 Since $\widetilde\rho_0$ has compact support in $\R^N$ and since $F$ is bounded, we see from \eqref{e2.17} that there exists a closed ball $K_R\subset\R^N$, centred at zero and radius $R\ge 1$, 
such that
\begin{equation}
\label{e2.10primo}
\rho_{M,l}(t,(\cdot,y))=0\quad\mbox{\rm on}\, \R^N\setminus K_R\;\mbox{\rm for all}\,(t,y)\in [0,T]\times E;\,M,l\in\N.
\end{equation}
 Furthermore, rewriting \eqref{e2.17} as \eqref{e2.4a} one easily sees that  for all $t\in[0,T]$
\begin{equation}
\label{e2.18'}
\int_{\R^N}\rho_{M,l}(t,(x,y))\,\Psi^2_{M,l}(x,y)\,dx=\int_{\R^N}\rho_{0}(x)\,\Psi^2_{M,l}(x,y)\,dx.
 \end{equation}
Below all statements are claimed to  hold for $\nu$-a.e., $y\in E$.

We need a few further lemmas of which the first is the most crucial, to prove Case 1. 
\begin{Lemma}
\label{l2.3n}
Let  $\epsilon>0.$ Then for all $1 \le i\le N,$ $l,M\in\N$
\begin{equation}
\label{e2.20n}
\begin{array}{l}
\ds \int_{\R^N}\left(\exp\left[\epsilon\left|  \left( \frac{\partial \Psi_{M,l}^2}{\partial x_i}\,/\Psi_{M,l}^2 \right)(x,y)  \right|    \right]-1\right)\,\Psi_{M,l}^2(x,y)\,dx\\
\\
\ds \le \int_{\R^N}\left(\exp\left[\epsilon\left|  \left(\frac{\partial \Psi_{M}^2}{\partial x_i}/\Psi_{M}^2\right)(x,y)  \right|    \right]-1\right)\,\Psi_{M}^2(x,y)\,dx\\
\\
\ds \le \int_{\R^N}\left(\exp\left[\epsilon\left|  \beta_{e_i}(x,y)  \right|    \right]-1\right)\,\Psi^2(x,y)\,dx.
\end{array} 
\end{equation}

\end{Lemma}
\begin{proof}
Obviously, the left hand side of \eqref{e2.20n}   is equal to
\begin{equation}
\label{e2.21n}
\int_{\R^N}\left(\exp\left[\epsilon\left|\int_{\R^N}\left(  \frac{\partial \Psi_{M}^2}{\partial x_i}\,  /\Psi_{M}^2\right)(\tilde x,y)     \,\Psi_{M}^2(\tilde x,y)\,\delta_l(x-\tilde x)\,d\tilde x \,(\Psi_{M,l}^2(x,y))^{-1}\right|\right]-1\right) \Psi_{M,l}^2(x,y))dx.
 \end{equation}
 Taking the modulus under the integral and
 applying Jensen's inequality for fixed $x\in\R^N$ to the probability measure
 $$
 \Psi_{M,l}^2(x,y))^{-1}\,\Psi_{M}^2(\tilde x,y) \,\delta_l(x-\tilde x)\,d\tilde  x
 $$
and the convex function $r\mapsto e^{\epsilon r}-1,\,r\ge 0$, we obtain that  \eqref{e2.21n} is dominated by
$$
\int_{\R^N}\int_{\R^N}\left(\exp\left[\epsilon\left(\left|  \frac{\partial \Psi_{M}^2 }{\partial x_i}\,\right| /\Psi_{M}^2\right)(\tilde x,y)\right]-1\right)     \,\Psi_{M}^2(\tilde x,y)\,\delta_l(x-\tilde x)\,d\tilde x \,dx.
$$
By Young's inequality and since $\|\delta_l\|_{L^1(\R^N)}=1,$ the latter is dominated by
\begin{equation}
\label{e2.22n}
\int_{\R^N}\left( \exp\left[\epsilon\left(\left|  \frac{\partial \Psi_{M}^2}{\partial x_i}\, \right| /\Psi_{M}^2\right)( x,y)\right] -1\right)    \,\Psi_{M}^2(x,y) \,dx.
\end{equation}
Hence the fist inequality in \eqref{e2.20n} is proved. To show the second we note that
$$
\frac{\partial \Psi_{M}^2 }{\partial x_i}(\cdot,y) ={\mathds 1}_{\{M^{-1}<\Psi^2(\cdot,y)<M\}}\,\frac{\partial \Psi^2}{\partial  x_i}(\cdot,y),\quad dx\mbox{\rm--a.s.}.
$$
Hence the integral in \eqref{e2.22n} is dominated by
$$
\int_{\R^N} {\mathds 1}_{\{M^{-1}<\Psi^2(\cdot,y)<M\}}\left(\exp\left[\epsilon \left(\left|  \frac{\partial  \Psi^2 }{\partial x_i}\,\right| /\Psi^2\right)( x,y)\right]-1\right)     \,\Psi^2(x,y) \,dx,
$$
which in turn by \eqref{e2.12n} is dominated by
the last integral in \eqref{e2.20n}
\qed
\end{proof}
\begin{Lemma}
\label{l2.4n}
For $\delta>0$ let $C_F(\delta)$ and $C_F(\delta,y)$ be as in Hypothesis \ref{h2}(iii).
Then  for
$$
\delta:=\inf_{1\le i\le N}\,\frac{c_i}{N( \|f_i\|_\infty+1)},
$$
we have 
 $$
\begin{array}{l}
\ds C_F(\delta,y)\le\sup_{M,l\in\N}\int_0^T\int_{\R^N} \Bigg(\exp\left[-\delta\sum_{i=1}^N \partial_{e_i} f(t,x)\right] ^+\\
\\
\ds\times \exp \left[\delta\sum_{i=1}^N \|f_i\|_\infty\left(\left|  \frac{\partial \Psi_{M,l}^2}{\partial x_i}\, \right| \Psi_{M,l}^2\right)( x,y)\right] -1\Bigg)    \,\Psi_{M,l}^2(x,y) \,dx\,dt<\infty
\end{array} 
$$
and $C_F(\delta)<\infty$.
\end{Lemma}
\begin{proof}
By \eqref{e2.12n}, \eqref{e2.13n} and
convexity of the function $r\mapsto be^{ar}-1,\,r\ge 0$, for $a,b>0$,
this follows immediately from Lemma \ref{l2.3n} and \eqref{e1.8primo}.
\qed
\end{proof}
\begin{Lemma}
\label{l2.5n}
(i) We have for all $M\in\N$, $t\in[0,T]$
$$
\lim_{l\to\infty}D^*_{M,l}F(t,(x,y))=-\sum_{i=1}^N\left[\partial_{e_i}f_i(t,x)+f_i(t,x)\left(  \frac{\partial \Psi_{M}^2}{\partial x_i}\, /\Psi_{M}^2\right)( x,y)  \right] =:D^*_{M}F(t,(x,y)), 
$$
and
$$
\lim_{M\to\infty}D^*_{M}F(t,(x,y))=-\sum_{i=1}^N  \left[\partial_{e_i}f_i(t,x)+f_i(t,x)\beta_{e_i}( x,y)  \right] =D^*_{x}F(t,(x,y)),
$$
in $L^1_{loc}(\R^N,dx)$.

(ii) Let $\rho_M$ and $\rho$ be defined as $\rho_{M,l}$ with $D^*_{M,l}F$ replaced by $D^*_{M}F$  and $D^*_{x}F$ respectively. 

Then there exist subsequences $(l_k)_{k\in\N}$, $(M_k)_{k\in\N}$ such that  we have for $dx$--a.e. $x\in \R^N$,  for all $M\in \N$
$$
\lim_{k\to\infty}\rho_{M,l_k}(t,(x,y))=\rho_{M}(t,(x,y)),\quad\forall\;t\in[0,T]
$$
and
$$
\lim_{k\to\infty}\rho_{M_k}(t,(x,y))=\rho(t,(x,y)),\quad\forall\;t\in[0,T].
$$
\end{Lemma}
\begin{proof}
(i) Obviously, for  all $M\in\N$ by \eqref{e2.15'}
$$
\lim_{l\to\infty}D^*_{M,l}F(t,(\cdot,y))=D^*_{M}F(t,(\cdot,y)),\quad\mbox{\rm in $L^1_{loc}(\R^N,dx)$},\,\forall\;t\in[0,T].
$$
 The second assertion follows, because
 \begin{equation}
\label{e2.21'}
\left(  \frac{\partial \Psi_{M}^2}{\partial x_i}\, /\Psi_{M}^2\right)( x,y) ={\mathds 1}_{(M^{-1},M)} \,  (\Psi^2( x,y))\left( \frac{\partial \Psi^2}{\partial x_i}\, /\Psi^2\right)( x,y)  .
 \end{equation}

 (ii) Fix $t\in[0,T]$. Then for all $u\in[0,t]$
 $$
 x\mapsto \xi(T-u,T-t,x)
 $$
 is a $C^1$--diffeomorphism on $\R^N$.
 Let $\phi_{u,t}:\R^N\to\R^N$ be its inverse (i.e.  just the corresponding backward flow).
 Then for every $K\subset\R^N$, $K$ compact, and $\Delta D^*_{M,l}F:=|D^*_{M}F-D^*_{M,l}F|$ we have
 $$
 \begin{array}{l}
\ds \int_K\int_0^t\Delta D^*_{M,l}F(T-u,(\xi(T-u,T-t,x),y)\,du\,dx\\
\\
\ds=  \int_0^t\int_{\xi(T-u,T-t,K)}\Delta D^*_{M,l}F(T-u, (x,y))\,|\det D\phi_{u,t}(x)|\,dx\,du.
\end{array} 
 $$
 Since $F$ is bounded, there exists a ball $B_R(0)$ so that for large enough $R>0,$
 $\xi(T-u,T-t,K)\subset B_R(0)$ for all $t\in[0,T]$. Hence by Fubini's Theorem the above integral is dominated by
 \begin{equation}
 \label{e2.15bis}
\int_{B_R(0)} \int_0^t |\det D\phi_{u,t}(x)| \Delta D^*_{M,l}F(T-u, (x,y))\,\,dx\,du.
 \end{equation}
 The specific dependence of $F$ on $T-u$ and the well known explicit formula of $\det D\phi_{u,t}$ (recall $\phi_{u,t}$ is a flow)     implies that
 $$
 x\mapsto \int_0^t |\det D\phi_{u,t}(x)| \,\widetilde f_i(T-u,x)\,du
 $$
 is locally bounded on $\R^N$, so that  \eqref{e2.15'}  can be applied to show that the term in \eqref{e2.15bis} converges to zero as $l\to\infty$ . So, the first assertion follows. Then also the second assertion follows by  \eqref{e1.14m}, \eqref{e2.21'} and the same arguments.
\qed
\end{proof}

 \begin{Lemma}
\label{l2.6n}
Let  $l,M\in\N$. Then for all $t\in[0,T]$ and $\delta>0$
\begin{equation}
\label{e2.24n}
\begin{array}{l}
\ds \int_{\R^N}\rho_{M,l}(t,(x,y))\left(  \ln \rho_{M,l}(t,(x,y))-1 \right) \Psi^2_{M,l}(x,y)\,dx\\
\\
\ds\le e^{t/\delta}\Bigg[\int_{\R^N}\rho_0(x)|\ln \rho_0(x)-1|\Psi^2_{M,l}(x,y)\,dx
\ds +C_F(\delta,y)\\
\\
\ds+\frac{t}\delta|\ln \delta|\int_{\R^N}\rho_0(x) \Psi^2_{M,l}(x,y)\, \,dx +\frac{t}{M}|K_{R+1}|+t\int_{\R^N}\Psi^2(x,y)\,dx\Bigg]
\end{array} 
\end{equation}
where $C_F(\delta,y)$ is as defined in Hypothesis \ref{h2}(iii) and $|K_{R+1}|$ denotes the Lebesgue measure of the ball $K_{R+1}\subset \R^N$, centred at $0$ and radius $R+1$, where $R$ is as in \eqref{e2.10primo}.
\end{Lemma}
\begin{proof}
Since $\rho_{M,l}(t,(\cdot,y)$ has compact support in $\R^N$ for all $(t,y)\in [0,T]\times E$ by the regularity properties of  $\rho_{M,L}$ stated in Corollary \ref{cA2} of  Appendix A, all integrals below are well defined. Since $M,l\in\N$ and  $y\in E$ are fixed, for simplicity of notation we denote the maps
$ x\mapsto \rho_{M,l}(t,(x,y))$
and $ x\mapsto \Psi_{M,l}(x,y)$ by $\rho(t)$, $\Psi$ respectively.
Then for $t\in[0,T]$,
$$
\begin{array}{l}
\ds \int_{\R^N}\rho(t)(\ln \rho(t)-1)\, \Psi^2dx\\
\\
\ds= \int_{\R^N}\rho_0(\ln \rho_0-1)\, \Psi^2dx+\int_{\R^N}\int_0^t\tfrac{d}{ds}[\rho(s)(\ln \rho(s)-1)]\,ds\, \Psi^2dx\\
\\
\ds=\int_{\R^N}\rho_0(\ln \rho_0-1)\, \Psi^2dx+\int_{\R^N}\int_0^t\ln \rho(s)D_s\rho(s)\,ds\, \Psi^2dx\\
\\
\ds= \int_{\R^N}\rho_0(\ln \rho_0-1)\, \Psi^2dx -\int_0^t\int_{\R^N} \langle F(s,x),D_x(\rho(s)(\ln\rho(s)-1))  \rangle \, \Psi^2dx\,ds\\
\\
\ds+\int_0^t\int_{\R^N}  D^*_{M,l}\,F(s,(\cdot,y))\rho(s)\,\ln\rho(s)    \,\Psi^2dx\,ds\\
\\
\ds=\int_{\R^N}\rho_0(\ln \rho_0-1)\,\Psi^2dx+\int_0^t\int_{\R^N}  D^*_{M,l}\,F(s,(\cdot,y))\rho(s)    \,\Psi^2dx\,ds\\
\\
\ds\le \int_{\R^N}\rho_0(\ln \rho_0-1)\, \Psi^2dx+\int_0^t\int_{\R^N}   \left[  e^{\delta(D^*_{M,l}\,F(s,(\cdot,y))^+}-1 +\tfrac1\delta\,\rho(s)\,(\ln (\tfrac1\delta\, \rho(s))-1)\right] \,\Psi^2dx\,ds\\
\\
\ds +t\int_{K_R}\Psi^2(x,y)\,dy,
\end{array} 
$$
where in the third equality we used \eqref{e2.19n}, in the fourth equality we used Fubini's theorem and the definition of  $D^*_{M,l}$   and finally, in the last inequality we used \eqref{e2.10primo}  and that $ab\le e^a+b(\ln b-1)$ for $a,b\ge 0$.
Now the assertion follows by  Gronwall's lemma, since by \eqref{e2.18'}
\begin{equation}
\label{e2.25n}
\int_{\R^N}\rho_{M,l}(t,(x,y))\,\Psi_{M,l}^2(x,y)\, dx=\int_{\R^N}\rho_0(x)\,\Psi_{M,l}^2(x,y)\, dx,\quad \forall\; \in[0,T],
 \end{equation}
and since 
 $$
\int_{K_R} \Psi_{M,l}^2(x,y)\,dx\le \frac1{M}\,|K_{R+1}|+\int_{\R^N} \Psi^2(x,y)\,dx.
 $$
 \qed
\end{proof}

\begin{Lemma}
\label{l2.7n}
Let  $M\in\N$, $\rho_{M,l,y}(t,x):=\rho_{M,l}(t,(x,y)),\; t\in [0,T]$, $x\in\R^N$, and $\Psi_{M,l,y}(x):=\Psi_{M,l}(x,y)$, $x\in\R^N$. Then $\{\rho_{M,l,y}\cdot \Psi^2_{M,l,y}:\,l\in\N   \}$ is uniformly integrable with respect to the measure $\chi(x)\,dx\,dt$, where $\chi$ is the indicator function of an arbitrary compact set in $\R^N$.
\end{Lemma}
\begin{proof}
Let $c\in(1,\infty)$. Then for all $l\in\N$ and $\rho_{l}:=\rho_{M,l,y}$, $\Psi_{l}:=\Psi_{M,l,y}$,
$$
\begin{array}{l}
\ds \int_0^T\int_{\R^N}{\mathds 1}_{\{\rho_l\Psi_l^2\ge c\}}\,\rho_l\Psi_l^2\,\chi\,dx\,dt\le \tfrac1{\ln c} \int_0^T\int_{\R^N}{\mathds 1}_{\{\rho_l\Psi_l^2\ge c\}}\,(\ln \rho_l+\ln\Psi_l^2)\rho_l\Psi_l^2\,\chi\,dx\,dt\\
\\
\ds\le \tfrac1{\ln c} \int_0^T\int_{\R^N} |\rho_l\ln \rho_l|\Psi_l^2\, dx\,dt+\tfrac{\ln(M+1)}{\ln c} \int_0^T\int_{\R^N}\rho_l\Psi_l^2\, dx\,dt.
\end{array} 
$$
Since $r\ln r-r\ge -1$, $r\in[0,\infty)$,   it follows by Lemma \ref{l2.6n} and \eqref{e2.25n}, that both integrals on the right hand side of the last inequality are uniformly bounded in $l$ and the assertion follows.
\qed
\end{proof}

Now we proceed with the proof of Case  1 of Theorem \ref{t2.1n}. It follows by \eqref{e2.19n} (analogously to  \eqref{e2.4a}--\eqref{e2.3a} above) that for all
\begin{equation}
\label{e2.26n}
u(t,x):=g(t)f(x),\quad t\in [0,T],\, x\in \R^N,
 \end{equation}
$g\in C^1([0,T];\R)$ with $g(T)=0$ and $f\in C^1_0(\R^N)$
that
\begin{equation}
\label{e2.27n}
\begin{array}{l}
\ds  \int_0^T\int_{\R^N} \left[D_tu(t,x)+
\langle D_xu(t,x), F(t,x)   \rangle 
   \right]\rho_{M,l}(t,(x,y))\,\Psi_{M,l}^2(x,y)\,dx\,dt\\
   \\
   \ds=-\int_{\R^N}u(0,x)\rho_0(x)\,\Psi_{M,l}^2(x,y)\,dx.
\end{array} 
\end{equation}
By Lemma \ref{l2.5n}(ii) and Lemma \ref{l2.7n} we can pass to the limit in
\eqref{e2.27n} along the subsequence $(l_k)_{k\in\N}$ from Lemma \ref{l2.5n} to conclude that for such $u$
\begin{equation}
\label{e2.28n}
\begin{array}{l}
\ds  \int_0^T\int_{\R^N} \left[D_tu(t,x)+
\langle D_xu(t,x), F(t,x)   \rangle 
   \right]\rho_{M}(t,(x,y))\,\Psi_{M}^2(x,y)\,dx\,dt\\
   \\
   \ds=-\int_{\R^N}u(0,x)\rho_0(x)\,\Psi_{M}^2(x,y)\,dx.
\end{array} 
\end{equation}
We can also pass to the limit in \eqref{e2.25n} to get
\begin{equation}
\label{e2.29n}
\int_{\R^N}\rho_M(t,(x,y))\,\Psi_M^2(x,y)\, dx=\int_{\R^N}\rho_0(x)\,\Psi_M^2(x,y)\, ,dx,\quad \forall\; t\in[0,T].
 \end{equation}
 Furthermore, by Lemma \ref{l2.5n}(ii) and Lemma \ref{l2.6n}  we deduce from \eqref{e2.24n} by Fatou's lemma that for all $t\in[0,T]$, $\delta>0$
 \begin{equation}
\label{e2.30n}
\begin{array}{l}
\ds \int_{\R^N}\rho_M(t,(x,y))(\ln \rho_M(t,(x,y))-1)\,\Psi_M^2(x,y)\,dx\\
\\
\ds\le e^{t/\delta}\Big[\int_{\R^N}\rho_0(x)|\ln \rho_0(x)-1|\,\Psi_M^2(x,y)\,dx+C_F(\delta,y)+\tfrac{t}\delta|\ln \delta|\int_{\R^N}\rho_0(x) \,\Psi_M^2(x,y)\,dx\\
\\
\ds+ \frac{t}{M}\,|K_{R+1}|+t\int_{\R^N} \Psi^2(x,y)\,dx\Big].
\end{array} 
\end{equation}

Taking now the subsequence $(M_k)_{k\in\N}$ from Lemma \ref{l2.5n} instead of $M$ and using  exactly analogous arguments as above, we can pass to the limit in \eqref{e2.28n}, \eqref{e2.29n} and \eqref{e2.30n} to obtain that for all $u$ as in \eqref{e2.26n}
\begin{equation}
\label{e2.31n}
\begin{array}{l}
\ds  \int_0^T\int_{\R^N} \left[D_tu(t,x)+
\langle D_xu(t,x), F(t,x)   \rangle 
   \right]\rho(t,(x,y))\,\Psi^2(x,y)\,dx\,dt\\
   \\
   \ds=-\int_{\R^N}u(0,x)\rho_0(x)\,\Psi^2(x,y)\,dx,
\end{array} 
\end{equation}
and for all $t\in[0,T]$
\begin{equation}
\label{e2.32n}
\int_{\R^N}\rho(t,(x,y))\,\Psi^2(x,y) \,dx=\int_{\R^N}\rho_0(x)\,\Psi^2(x,y)\,dx,
 \end{equation}
 and for all $t\in[0,T]$, $\delta>0$
 \begin{equation}
\label{e2.33n}
\begin{array}{l}
\ds \int_{\R^N}\rho(t,(x,y))(\ln \rho(t,(x,y))-1)\,\Psi^2(x,y)\, dx\\
\\
\ds\ds\le e^{t/\delta}\Big[\int_{\R^N} \rho_0(x)|\ln \rho_0(x)-1|\,\Psi^2(x,y) \,dx+C_F(\delta,y)+\tfrac{t}\delta|\ln \delta|\int_{\R^N}\rho_0(x) \,\Psi^2(x,y)\,dx\\
\\
\ds + t\int_{\R^N} \Psi^2(x,y)\,dx \Big].
\end{array} 
\end{equation}
Taking the special  $\delta$ from Lemma \ref{l2.4n} and 
$C_F(\delta,y)$  as in Lemma \ref{l2.5n}  in the situation of Case 1 the assertion of Theorem \ref{t2.1n} now follows easily from the disintegration formula \eqref{e2.10n}, integrating  \eqref{e2.31n} with respect  to $\nu$ and by approximating the functions $u$ in \eqref{e1.1} in the obvious way. From  \eqref{e2.33n} we get \eqref{e2.6n} after  integrating   over $y$ with respect  to $\nu$.
\qed
\end{proof}
\begin{Remark}
\label{r2.8}
\em (i) We here emphasize that in the situation of Case 1  we have an explicit formula for the solution density in  \eqref{e2.31n} given by 
\begin{equation}
\label{e2.34n}
\rho(t,(x,y))=\rho_0 (\xi(T,T-t,x))e^{-\int_{0}^{t}D_x^*F(T-u 
,\xi(T-u ,T-t,y))du }
 \end{equation}
for $t\in[0,T]$ and $dx$--a.e. $x\in \R^N$ with $\xi$ given as in Corollary \ref{cA2} of Appendix A.

(ii) Integrating \eqref{e2.33n} over $y\in E$ with respect to $\nu$, from Lemma \ref{l2.2n} we obtain that for all $t\in[0,T]$, $\delta>0$
 \begin{equation}
\label{e2.35n}
\begin{array}{l}
\ds \int_{H}\rho(t,x)(\ln \rho(t,x)-1)\,\gamma(dx) \le e^{t/\delta}\left[\int_{H} \rho_0|\ln \rho_0-1|\,d\gamma+C_F(\delta)+\tfrac{t}\delta|\ln \delta|\int_{H}\rho_0\,d\gamma+t\gamma(H)\right]
\end{array} 
\end{equation}
and likewise from  \eqref{e2.32n}
 that for  all  $t\in[0,T]$
\begin{equation}
\label{e2.35'n}
\int_{H}\rho(t,x)\,\gamma(dx)=\int_{H}\rho_0(x)\,\gamma(dx) =1.
 \end{equation}
 \end{Remark}\medskip

 {\bf Case 2}. Let $F_j,\,j\in\N,$ be as in Hypothesis \ref{h2}. Choose nonnegative $\rho_{0,j}\in \mathcal FC^1_0$ such that
 \begin{equation}
\label{e2.36n}
\lim_{j\to\infty}\rho_{0,j}=\rho_{0}\quad\mbox{\rm in}\;L^1(H,\gamma)
 \end{equation}
 and
  \begin{equation}
\label{e2.37n}
\sup_{j\in\N}\int_H\rho_{0,j}\,\ln \rho_{0,j}\,d\gamma<\infty. 
 \end{equation}
For existence of such $\rho_{0,j},\,j\in\N,$  see Corollary \ref{cC3} in Appendix C below.
 
 Let $\rho_j$ be the corresponding solutions to \eqref{e1.1} with  $F_j$ replacing   $F$ and   $\zeta:=\rho_0\cdot \gamma$, which exist by Case 1. Then by \eqref{e2.35n} with $\rho_j, F_j, \rho_{0,j}$ replacing $\rho, F$ and $\rho_0$ respectively, Hypothesis \ref{h2} and \eqref{e2.35'n} imply that
 \begin{equation}
\label{e2.38n}
\sup_{j\in \N}\,\sup_{t\in[0,T]}\int_H\rho_j(t,x)\,\ln \rho_j(t,x)\,\gamma(dx)<\infty.
 \end{equation}
 
 By Case 1 we have for all $u\in \mathcal FC^1_{b,T}$
 \begin{equation}
  \begin{array}{l}
\label{e2.8m}
\ds\int_0^T\int_H\left[\frac{d}{dt}\,u(t,x)+\langle D_xu(t,x),F_j(t,x)  \rangle_H\right]\,\rho_j(t,x)\,\gamma(dx)\,dt\\
\\
\ds =-\int_Hu(0,x)\,\rho_{0,j}(x)\gamma(dx).
\end{array}
\end{equation} 
So, by \eqref{e2.36n} we only have to consider the convergence of the left hand side of  \eqref{e2.8m}, more precisely only the part of it involving $F_j$. But
\begin{equation}
\label{e2.9m}
\begin{array}{l}
\ds \left|\int_0^T\int_H (\langle D_xu,F_j  \rangle_H\,\rho_j-\langle D_xu,F  \rangle_H\,\rho)\,d\gamma\,dt\right|\\
\\
\ds\le \|Du\|_\infty\int_0^T\int_H|F_j-F|_H\,\rho_j\,d\gamma\,dt+
\left|\int_0^T\int_H  \langle F,Du \rangle\,(\rho_j-\rho) \,d\gamma\,dt\right|
\end{array} 
\end{equation}
Because of the boundedness of $ \langle F,Du \rangle$ the second term on the right hand side of \eqref{e2.9m} converges to $0$ if $j\to\infty$. Let $\epsilon>0$. Then, by Young's inequality,
the first  term on the right hand side of \eqref{e2.9m}  is up to a constant dominated by
$$
\int_0^T\int_H e^{\frac1\epsilon|F_j-F|_H} \,d\gamma\,dt+\epsilon
\int_0^T\int_H \rho_j\ln(\epsilon\rho_j)\,d\gamma\,dt,
$$
of which the first summand converges to zero  as $j\to\infty$, since $F_j,F$ are uniformly bounded, while the second summand is dominated by
$$
\epsilon
\int_0^T\int_H \rho_j\ln\rho_j\,d\gamma\,dt+\epsilon\ln\epsilon,
$$
which can be made arbitrarily small uniformly in $j$ because of \eqref{e2.38n}. Hence putting all this together we conclude that the right hand side of \eqref{e2.9m} converges to $0$  as $j\to\infty.$  \eqref{e2.6n}
then follows by weak lower semi--continuity.
Finally from \eqref{e2.35'n} and \eqref{e2.36n} it follows that $\nu_t(dx):=\rho(t,x)\,\gamma(dx)$ is a probability measure for all $t\in[0,T]$. Thus Theorem \ref{t2.1n} is completely proved.
 \begin{Remark}
 \label{r2.9}
 \em
 Though the finite entropy condition in the initial measure $\rho_0$ is crucial in the proof of Theorem \ref{t2.1n}, it could be replaced by a corresponding assumption with $r\mapsto r( \ln r-1)$ replaced by another Young fnction (see Appendix C below) and adjusting Hypothesis \ref{h2}(ii) accordingly. In particular, we can take e. g. $r\mapsto r^p, \,r\ge 0,\,p>1$. Then the exponential integrability condition on $D^*_xF$ in Hypothesis \ref{h2}(iii) can be replaced by an $L^{p'}$-integrability condition with $p'=\tfrac{p}{p-1}.$  Hence the solution $\rho$ to \eqref{e1.2} would be in $L^p([0,T]\times H,dt\otimes \gamma)$, provided $\rho_0\in L^p(H,\gamma)$. Therefore, we get existence of solutions also in the situation of Section 3, provided $B$ in \eqref{e3.1} is the identity operator (see Corollary \ref{c3.11primo} below). Likewise, e.g. for the Young $r\to r^p,\,r\ge 0,\, p>1,$ one can relax the  assumption on exponential integrability on $\beta_h,\,h\in Y$, in Hypothesis \ref{h1} by $L^p(H,\gamma)$ integrability.
  \end{Remark}
  \begin{Example}
  \label{ex2.10}
  \em
Let us discuss Hypothesis \ref{h2}(ii) for $\gamma$ as in Example  \ref{ex1.3}(ii). In this case we choose $\{e_n:\,n\in\N\}$ to be the eigenbasis of $A$ given by
$$
e_n(\xi):=\sqrt{\frac2\pi}\,\sin(n\pi \xi),\quad\xi\in[0,1],\;n\in\N.
$$
Then for  $A_ne_n=-\lambda_n e_n$ with $\lambda_n:=\pi^2\,n^2$, $n\in \N$. Now consider the corresponding disintegration  \eqref{e2.10n}. Then   $N(0,\tfrac12\,(-A)^{-1})$ is by independence equal to the convolutions of his projections on $H_N$ and $E_N$ respectively. Hence
$$
\Psi_N^2(x,y)=\frac1{(2\pi\lambda_1\cdots\lambda_N)^{N/2}Z}\,\exp\left( -\frac{\alpha}{p}\,\int_0^1|x(\xi)+y(\xi)|^p\, d\xi\, -\frac14\sum_{i=1}^N\lambda_i^{-1}  \langle e_i,x  \rangle^2 \right)  
$$
where $y\in E_N$ and $x(\xi)= \langle   x,e_1\rangle e_1(\xi)+\cdots+ \langle x,e_n  \rangle e_n(\xi)$. 
So, obviously for $\nu_N$--a.e. $y\in E_N$,
$x\mapsto\Psi_N^2(x,y)$ is continuous and strictly positive on $H_N$, since $x+y\in L^p(0,1)=:L^p$, because  $N(0,\tfrac12(-A)^{-1})(C([0,1];\R))=1.$
Thus \eqref{e1.14m} holds.
Unfortunately so far we do not know whether \eqref{e1.14m} holds in case of $\gamma$ as in  \ref{ex1.3}--(iii).   
Now consider again the situation of  \ref{ex1.3}--(ii). We are now going to present a class $F:[0,T]\times H\to H$ for which  Theorem \ref{t2.1n} applies: Let  $f\in C_b([0,T]\times\R;\R)$ such that $ f(t,\cdot)\in C^1(\R;\R)$ for every $t\in[0,T]$ and there exist $K\in(0,\infty)$, $\delta\in (0,p)$ such that for  $f'(t,r)=f_r(t,r)$
$$
f'(t,r)\ge-K(1+|r|^2+\alpha|r|^{p-\delta}),\quad\forall\;(t,r)\in [0,T]\times\R.
$$
Define $F_0: [0,T]\times L^2(0,1)\to L^2(0,1)$ by
$$
F_0(t,x)(\xi):=f(t,x(\xi)),\quad \xi\in (0,1),\;t\in [0,T]]
$$
and  $F: [0,T]\times L^2(0,1)\to L^2(0,1)$ by
\begin{equation}
\label{e*}
F(t,x):=(-A)^{-1}F_0(t,x),\quad x\in L^2(0,1),\;t\in [0,T].
 \end{equation}
Now we want to check Hypothesis \ref{h2} for this type of $F$.

{\it Claim 1}. For every $\epsilon>0$ there exists $C_\epsilon\in(0,\infty)$ such that
$$
\sum_{i=1}^N\partial_{e_i}F^i(t,x)\ge-C_\epsilon-\epsilon(|x|^2_{L^2}+\alpha |x|^p_{L^p]}),\quad x\in L^p(0,1),\;t\in [0,T],\;N\in\N,
$$
where
$$
F^i(t,x):= \langle e_i, F(t,x)   \rangle.
$$
{\it Proof of Claim 1}. Let  $x\in L^p(0,1),\;t\in [0,T]$. Then
$$
\begin{array}{lll}
\ds \sum_{i=1}^N\partial_{e_i}F^i(t,x)&=&\ds
\sum_{i=1}^N\lambda_i^{-1}\partial_{e_i}\int_0^1e_i(\xi)\,f(t,x(\xi))\,d\xi\\
\\
&=&\ds
\sum_{i=1}^N\lambda_i^{-1} \int_0^1e_i^2(\xi)\,f'(t,x(\xi))\,d\xi\\
\\
&\ge&\ds-K
\sum_{i=1}^\infty\lambda_i^{-1} \int_0^1e_i^2(\xi)\,(1+|x(\xi)|^2+\alpha|x(\xi)|^{p-\delta})\,d\xi\\
\\
&\ge&\ds-C_\epsilon-\epsilon(|x(\xi)|_{L^2}^2+\alpha|x(\xi)|_{L^p}^p)
\end{array} 
$$
by YoungÕs inequality.\hfill $\Box$\medskip

{\it Claim 2}. For every $\epsilon>0$ there exists $C_\epsilon\in(0,\infty)$ such that
$$
\sum_{i=1}^N\beta_i(x)F^i(t,x)\ge -C_\epsilon-\epsilon(|x(\xi)|_{L^2}^2+\alpha|x(\xi)|_{L^p}^p),\quad \forall\,x\in L^p(0,1),\,t\in[0,T],\,N\in\N.
$$
{\it Proof of Claim 2}. Let  $x\in L^p(0,1),\;t\in [0,T]$. Then by \eqref{e1.3a}    
$$
\begin{array}{lll}
\ds\sum_{i=1}^N\beta_i(x)F^i(t,x)&\ge&\ds -\sum_{i=1}^N\int_0^1e_i(\xi)\,x(\xi)\,d\xi\,\int_0^1e_i(\xi)\,f(t,x(\xi))\,d\xi\\
\\
&&\ds-\alpha\sum_{i=1}^N\lambda_i^{-1}\int_0^1e_i(\xi)\,|x(\xi)|^{p-2}\,x(\xi)\,d\xi\,\int_0^1e_i(\xi)\,f(t,x(\xi))\,d\xi\\
\\
&\ge&\ds  -\langle P_NF_0(t,x),P_Nx\rangle- \alpha\sum_{i=1}^\infty\lambda_i^{-1}\,|f|_\infty\,\sqrt{\tfrac2\pi}\,|e_i|_{L^p}\,\big||x|^{p-1}\big|_{L^{p/(p-1)}} \\
\\
&\ge&\ds  - |F_0(t,x)|_{L^2}\,|x|_{L^2}- \alpha\sum_{i=1}^\infty\lambda_i^{-1}\,|f|_\infty\,\tfrac2\pi|x|^{p-1}_{L^p}  \\
\\
&\ge&\ds-C_\epsilon-\epsilon(|x|^2_{L^2}+\alpha |x|^p_{L^p}),
\end{array} 
$$
where $P_N$ denotes the orthogonal projection in $L^2(0,1)$ onto $H_N$, i.e. the linear span of $\{e_1,...,e_N\}$.  $\Box$

We note that $C_\epsilon$ can be taken in both claims to be a function only on $\delta, K$ and  $|f|_\infty$ which is increasing in $K$ and $|f|_\infty$, while decreasing in $\delta$.

Now let us prove that by Claim 1 and Claim 2 that Hypothesis \ref{h2} is satisfied. To avoid a further regularization procedure let  us additionally assume that $f(t,\cdot)\in C^2(\R)$ for all $t\in[0,T]$ and $\frac{\partial}{dr}\,f(t,\cdot), \frac{\partial^2}{dr^2}\,f(t,\cdot)\in C([0,T]\times \R)$. Define for $j\in\N,\,x\in H,$ $t\in [0,T]$
\begin{equation}
\label{e**}
F_j(t,x):=P_jF(t,P_jx)=\sum_{i=1}^j\left(\lambda_i^{-1}\int_0^1e_i(\xi)\,f(t,(P_jx)(\xi))\,d\xi   \right)\,e_i,
 \end{equation}
 where $P_j$  is the orthogonal projection onto the linear span of $\{e_1,...,e_j\}$ in $H=L^2(0,1)$. Then obviously $F_j$ is as in Hypothesis \ref{h2}(iii) with $N_j=j$ and
 $$
 \widetilde f_i(t,x_1,...,x_j)=\lambda_i\int_0^1e_i(\xi)f\left(t,\sum_{l=1}^j x_le_l(\xi)\right)\,d\xi,
 $$
for $(x_1,...,x_j)\in\R^j.$ Now let us consider
the corresponding $C_{F_j}(\delta)$ from Hypothesis \ref{h2}(iii) and  $\Psi_N$ defined in
Lemma \ref{l2.4n}. Note that $\Psi_N^2(\cdot,y)$ above is $C^2$ and strictly positive on $H_j=\R^j$ for $\nu_N$-a.e. $y\in E.$ Hence by definition $\Psi_{N,M,l}^2=\Psi_N^2$ for all $M,l\in\N$. Hence for $(x,y)\in H_j\oplus H^\perp_j$, $t\in [0,T]$ by Claim 1 and Claim 2
$$
D^*_{N_j,M,l}F_j(t,(x,y))\le C_\epsilon+\epsilon(|(x,y)|^2_{L^2}+\alpha |(x,y)|^p_{L^p}).
$$
Here we used that $\|P_j\|_{L^p\to L^p}\le c_p\in(0,\infty) $ which is  independent of $j$ (see e.g. \cite[Section 2C16]{LiTr79}). Hence obviously for $\delta\in (0,1)$
$$
\sup_{j\in\N} C_{F_j}(\delta)<\infty.
$$
Hence by  Theorem \ref{t2.1n} we have a solution
$$
\nu_t(dx)=\rho(t,x)\gamma(dx),\quad t\in[0,T],
$$
with $\gamma$ as above, for equation  \eqref{e1.1} for $F$ as above with initial condition $\rho_0\,\gamma$ with $\rho_0$ in $L\log L$ with respect to  $\gamma$.\medskip

Now we shall prove that this solution is also unique provided $\alpha>0$, so 
$\gamma$ is not Gaussian. We shall, however, apply a uniqueness result for the Gaussian reference measure $N(0,\tfrac12(-A)^{-1})$ proved in \cite{DaFlRoe14}, because $\nu_t$ has the density
$$
\bar\rho(t,x)=\rho(t,x)\tfrac1{Z}\,e^{-\frac\alpha{p}|x|^p_{L^p}},\quad (t,x)\in [0,T]\times H
$$
with respect to $N(0,\tfrac12(-A)^{-1})$. Let us first show that $\bar\rho$  is bounded in $(t,x)$. To this end we first note that because $\sum_{i=1}^\infty  \lambda_i^{-1}<\infty$,
$$
R:=\sup_{j\in\N}\,\left\||F_j|_{L^p}\right\|_\infty<\infty.
$$
Hence the corresponding flows $\xi_j$ from
 \eqref{eA1} with $F_j$  replacing  $F$ will all stay in the $L^p$ ball $B^p_{TR}(x)$ for all times in $[0,T]$ when started at
  $x$ in  $L^p(0,1)$. This implies by Claim 1 and 2 that the exponent of the density $\rho^j $ in \eqref{e2.34n} with $F_j$  replacing $F$ will also have an upper bound of type
 $$
 C_\epsilon+\epsilon(|x|_{L^2}^2+\alpha|x|_{L^p}^p),\quad \forall\,x\in L^p(0,1)
 $$
independent of $j$. Hence it follow that
$$
\bar\rho^j(t,x):=\rho^j(t,x)\tfrac1{Z}\,e^{-\frac\alpha{p}|x|^p_{L^p}},\quad (t,x)\in [0,T]\times H
$$
is $N(0,\tfrac12(-A)^{-1})$--essentially bounded, uniformly in $j$, hence  so is its a.e. limit $\bar\rho$.

Now we can apply Theorem 2.3 in \cite{DaFlRoe14} for $p=\infty$ (which by a misprint there, seems to be excluded, but is in fact included in that theorem) to conclude uniqueness if we can prove the following properties (a)--(c) of $F$ defined above. For this we additionally assume: 
\begin{equation}
\label{e***}
\mbox{\it There exists $C,M\in(0,\infty)$ such that}\; |f'(t,r)|\le C(1+|r|^M),\quad r\in \R.
 \end{equation}
 
 (a) $F([0,T]\times H)\subset  (-A)^{-1/2}(H)$.\medskip
 
 (b) There exists $s\in (1,\infty)$ such that
 $$
 \int_0^T\int_H|(-A)^{1/2}F(t,x)|_H^s\,\gamma_0(dx)\,dt<\infty,
 $$

 (c) $F\in L^2(0,T;W^{1,s}(H;H,\gamma_0)$, which is defined as the closure of all vector  fields $F([0,T]\times H)\to H$ of type \eqref{e2.7n} with respect to the nom
 $$
 \|F\|_{1,s,T}:=\left(\int_0^T\int_H(\|DF(t,x)\|^s_{\mathcal L_2(H)} +|F(t,x)|_H^2 ) \, \gamma_0(dx)\,dt\right)^{1/s},
 $$
 where $\|\cdot\|_{\mathcal L_2(H)}$ denotes the Hilbert--Schmidt norm and 
  $\gamma_0=N(0,\tfrac12(-A)^{-1})$.\medskip
  
  By the definition of $F$ in  \eqref{e*} property (a) obviously holds. (b) holds for all $s\in (1,\infty)$ since
  $$
  |(-A)^{1/2}F(t,x)|_H= |(-A)^{-1/2}F_0(t,x)|_H\le const. \|f\|_\infty.
  $$
 So, let us check (c):  Let $F_N$ be  as in 
 \eqref{e**}. Then for $1\le i,j\le N$
 $$
 \partial_{e_j} \langle e_i, F_N(t,x)   \rangle =\frac1{\lambda_i}\int_0^1e_i(\xi)\,e_j(\xi) f'(t,(P_Nx)(\xi))\,d\xi,\quad (t,x)\in [0,T]\times H.
 $$
Hence by  \eqref{e***} for some constant $c_1\in (0,\infty)$
$$
\begin{array}{l}
\ds\|DF_N(t,x)\|_{\mathcal L_2(H)}^2=\sum_{i=1}^N\frac1{\lambda_i}\int_0^1e^2_i(\xi)\,|f'(t,(P_Nx)(\xi))|^2\,d\xi\\
\\
\ds\le C_1\sum_{i=1}^\infty\frac1{\lambda_i}\,\sup_{N\in\N}\,\|P_N\|^{2M}_{L^{2M}\to L^{2M}}\,\left( 1+|x|^{2M}_{L^{2M}}  \right).  
\end{array} 
$$
Hence $F_N(t,x)),\,N\in \N,$ is bounded in the norm $\|\cdot\|_{1,2,T}$. Since $\sup_{n\in\N}\|F_N\|_\infty<\infty$ and $F_N\to F$ $dt\otimes \gamma_0$--a.e., (c) follows for $s=2$, because the operator $D$ is closable.

\end{Example}

  \section{Uniqueness}
 In Example \ref{ex2.10} of previous section we proved uniqueness for \eqref{e1.2} using the uniqueness result from \cite{DaFlRoe14} for Gaussian reference measures $\gamma$.
 For non--Gaussian, reference measures $\gamma$ uniqueness for \eqref{e1.2} is much more difficult to prove. In this section we do that for a whole class of  non Gaussian, reference measures $\gamma$. 
 
 \subsection{Notations and preliminaries}
 
 In this section, we   take as    reference  measure $\gamma$ the invariant measure   of the following reaction--diffusion   equation in $H:=L^2(0,1)$, 
\begin{equation}
\label{e3.1}
\left\{\begin{array}{lll}
dX(t)=[AX(t)+p(X(t))]dt+BdW(t),\\
\\
X(0)=x,\quad x\in H,
\end{array}\right.
\end{equation}
where  $A$ is  the realisation of the Laplace operator  $D^2_\xi $ equipped with
Dirichlet boundary conditions,
$$
Ax=D^2_\xi x,\quad x\in D(A),\quad D(A)=H^2(0,1)\cap H^1_0(0,1),
$$
  $p$ is a decreasing polynomial   
of odd degree equal to $N>1$,  $B\in L(H)$ with a {\em bounded inverse} and  $W$  is an $H$--valued cylindrical Wiener process  on a filtered
probability space
$(\Omega,\mathcal F, (\mathcal F_t)_{t>0},\P)$. 
Let us  recall the definition of   solution of \eqref{e3.1}.
\begin{Definition}
\label{d4.5}
(i). Let $x\in L^{2N}(0,1);$ we say that $X\in C_W([0,T];H)$
$\footnote{By $C_W([0,T];H)$ we  mean the set of  $H$--valued stochastic processes continuous in mean square and adapted to the filtration $(\mathcal F_t)$ .}$ is a
 {\em mild} solution
of problem \eqref{e3.1} if $X(t)\in L^{2N}(0,1)$ for all $t\ge 0$ and  fulfills the following integral equation
\begin{equation}
\label{e4.8}
X(t)=e^{tA}x+\int_0^te^{(t-s)A}p(X(s))ds+\int_0^te^{(t-s)A}dW(s),\quad t\ge 0.
\end{equation}

\noindent (ii). Let $x\in H;$ we say that $X\in C_W([0,T];H)$ is a {\em generalized} solution
of problem \eqref{e3.1} if there exists a sequence $(x_n)\subset L^{2N}(0,1),$ such that
$$
\lim_{n\to \infty}x_n=x\quad\mbox{\rm in}\;L^{2}(0,1),
$$
and
$$
\lim_{n\to \infty}X(\cdot, x_n)=X(\cdot,x)\quad\mbox{\rm in}\;C_W([0,T];H).
$$

\end{Definition}

It is convenient to introduce  the following approximating problem
\begin{equation}
\label{e4.9}
\left\{\begin{array}{l}
dX_\alpha (t)=(AX_\alpha (t)+p_\alpha (X_\alpha (t))dt+B\,dW(t),\\
\\
X_\alpha (0)=x\in H,
\end{array}\right.
\end{equation}
where for any $\alpha \in(0,1],$   $p_\alpha$  are the Yosida approximations of $p$, that is
$$
p_{\alpha }(r)= \frac{1}{\alpha }
\;(r-J_{\alpha}(r)),\;J_{\alpha }(r)=(1-\alpha p(\cdot))^{-1}(r),\quad r\in \R.
$$
Notice that, since $p_{\alpha }$ is Lipschitz continuous, then 
 for any $\alpha >0,$ and any $x\in H,$
  problem $(\ref{e4.9})$ has a unique 
 solution $X_\alpha (\cdot,x)\in C_W([0,T];H)$.

The following result is proved in \cite[Theorem 4.8]{Da04}
\begin{Proposition}
\label{p4.8} 
Let $T>0$, then 
\begin{enumerate}

\item[(i)] If $x\in L^{2N}(0,1),$  problem
  $(\ref{e3.1})$ has a unique  mild solution $X(\cdot,x)$.    
  
  \item[(ii)] If $x\in L^{2}(0,1),$  problem
  $(\ref{e3.1})$ has a unique  generalized solution $X(\cdot,x).$\medskip

 \noindent In both cases $\ds{\lim_{\alpha \to 0}X_\alpha (\cdot,x)=X(\cdot,x)}$ in 
  $C_W([0,T];H).$

\end{enumerate}

\end{Proposition}

Let us introduce now the transition semigroups $P_t$ and $P_t^\alpha$, setting
\begin{equation}
\label{e1}
P_t\varphi(x)=\E[\varphi(X(t,x))],\quad \varphi\in B_b(H)
\end{equation} 
and  
$$
P_t^\alpha\varphi(x)=\E[\varphi(X_\alpha(t,x))],\quad \varphi\in B_b(H).
$$

This definition extends to vector fields:\ if $G:H\rightarrow H$ is measurable
bounded, we call $\left(  \mathbf{P}_{t}G\right)  \left(  x\right)  $
the element of $H$ such that
\[
\left\langle \left(  \mathbf{P}_{t}G\right)  \left(  x\right)
,h\right\rangle _{H}=\mathbb{E}\left[  \left\langle G\left(  X\left(
t,x\right)  \right)  ,h\right\rangle _{H}\right]
\]
for every $h\in H$. It exists since
\[
\left\vert \mathbb{E}\left[  \left\langle G\left(  X\left(  \epsilon,x\right)
\right)  ,h\right\rangle _{H}\right]  \right\vert \leq\mathbb{E}\left[
|G(t,x)|
_{H}\right]  \left |h\right|_{H}\leq C_{G}\,|h|
_{H}
\]
where $C_{G}$ bounds $G$. In the sequel we shall use the notation
\[
\left(  \frac{I-\mathbf{P}_{t}}{t}\right)  G\left(  t,x\right)
\]
for $\frac{G\left(  t,x\right)  -\left(  \mathbf{P}_{t}G\left(
t,\cdot\right)  \right)  \left(  x\right)  }{t}$ and for analogous
expressions. We shall use similar notations for the semigroups associate to
the Yosida regularizations, $P_{t}^{\alpha}$ and $\mathbf{P}_t^{\alpha}$.

Denote by $L_{2}\left(  H\right)  $ (resp. $\mathcal{L}\left(  H\right)
$)\ the Hilbert-Schmidt norm (resp. operator norm) of operators in $H$.

 The sequence $(e_j)$
\begin{equation}
\label{e3.2d}
e_j(\xi)=\sqrt{\tfrac2\pi}\;\sin(j\pi\xi),\quad \xi\in[0,1],\; j\in\N,
 \end{equation}
  is an orthonormal basis  in $H$ and it results
\begin{equation}
\label{e3.3d}
Ae_j=-\alpha_j e_j,\quad\forall\;j\in\N,
 \end{equation}
 where
 $$
 \alpha_j:=\pi^2\,j^2,\quad\forall\;j\in\N.
 $$
 \begin{Lemma}
 \label{l3.3j}
 For every $\theta_{0}>1/4$ we have $(-A)^{-\theta_{0}}\in L_2(H)$.
 \end{Lemma}
 \begin{proof}
 We have in fact
$$ |( -A) ^{-\theta_{0}}|_{L_{2}\left(
H\right)  }^{2}=\sum_{j\in\mathbb{N}}|(-A)^{-\theta_{0}}e_{j}| _{H}^{2}=\sum_{j\in\mathbb{N}}\left\vert
j\right\vert ^{-4\theta_{0}}<\infty.
$$
\qed
\end{proof}
In the sequel we denote by $\theta_{0}$ any number in $(\tfrac14,\tfrac12)$. We need $\theta_{0}<\tfrac12$ for the results on stochastic convolution.
 
\begin{Remark}
\em When $B$ is equal to the identity,   \eqref{e3.1} is a gradient system and the corresponding transition semigroup $P_t$ is symmetric whereas if $B\neq I$,   $P_t$ is not symmetric.

\end{Remark}

 For  $P^\alpha_t$
 the following  Bismut-Elworthy-Li formula holds, see \cite{ElLi94} and \cite{DaZa14}.
 \begin{equation}
\label{e4.10}
 \langle D_xP^\alpha _t\varphi(x),h  \rangle=\frac1t\;\E\left[\varphi(X_\alpha (t,x))\int_0^t\langle 
B^{-1}\eta_\alpha^h(s,x),dW(s) \rangle  \right],\quad h\in H, 
\end{equation}
 where for any $h\in H$,  $\eta_\alpha^h(t,x)=:D_xX_\alpha(t,x)\cdot h$  is the differential of $X_\alpha(t,x)$ with respect to $x$ in the direction $h$. $\eta_\alpha ^{h}(t,x)$ is the  solution of the following equation with random coefficients
 \begin{equation}
\label{e5b}
 D_t\eta_\alpha ^{h}(t,x)=A\eta_\alpha ^{h}(t,x)+D_xp_\alpha(X_\alpha(t,x))
 \eta_\alpha^{h}(t,x),\quad \eta_\alpha^{h}(0,x)=h.
\end{equation}
The proof of the following lemma is a straightforward consequence of the dissipativity  of $p(\cdot)$.
\begin{Lemma}
\label{l3.1a}
It results
\begin{equation}
\label{e3.10c}
|\eta^h_{\alpha}(t,x)|_H\le  |h|_H,\quad\forall\;t\ge 0,\,x,h\in H,\,\alpha\in(0,1].
\end{equation}

\end{Lemma}
   \begin{Proposition}
\label{p4.16} 
Semigroups $P_t$ and $P^\alpha_t$  have  unique invariant
measures $\gamma, \gamma^\alpha$ respectively. Moreover 
$ \gamma^\alpha$ is weakly convergent to  $\gamma$ and  for any $N\in\N$ there exists $c_N>0$
such that
 \begin{equation}
\label{e4.11}
\int_H|x|_{L^{2N}(0,1)}^{2N}\gamma^\alpha(dx)\le c_N,\quad \int_H|x|_{L^{2N}(0,1)}^{2N}\gamma (dx)\le c_N.
\end{equation}
\end{Proposition}
(see \cite[Proposition 4.20]{Da04} and \cite[Proposition 15]{DaDe14}).
\begin{Corollary}
\label{c3.6}
Let  $h(x)\in D(A)$--$\nu$--a.e.  $x\in H,$ and
$Ah\in L^4(H),\gamma)$. Then there exists $K>0$ such that
 \begin{equation}
\label{e3.11a}
 \int_H|Dp_\alpha(x)h(x)|^2\,\gamma(dx)\le K\|Ah\|^2_{L^4(H,\gamma)} ,\quad\forall\;\alpha\in(0,1].
\end{equation}
 \end{Corollary}
  \begin{proof}
 Let $h(x)\in D(A)$. Then there is $K_1>0$ such that
    $$
  |p'(x)h(x)|^2\le K_1 |x^{N-1}|^2\,|h(x)|^2_{D(A)}
  \le K_1 |x|^{2N-2}_{L^{2N-2}}\,|h(x)|^2_{D(A)}.
  $$
  Integrating with respect to $\gamma$ over $H$ and using H\"older's inequality, yields
  $$
  \begin{array}{lll}
  \ds\int_H|p'(x)h(x)|^2\,\gamma(dx)&\le&\ds K_1\int_H|x|^{2N-2}_{L^{2N-2}}\,|Ah(x)|^2\,\gamma(dx)\\
  \\
  &\le&\ds 
K_1\int_H|x|^{4N-4}_{L^{2N-2}}\,\gamma(dx)\,\|Ah\|^2_{L^4(H,\gamma)}.
\end{array}
  $$
Now the conclusion follows from \eqref{e4.11}.  
\qed  
\end{proof}

 Let us finally  recall the elementary  identity, see \cite{DaDe14}
\begin{equation}
\label{e3.9aa}
\langle P^{\alpha}_tD_x\varphi,h   \rangle=\langle D_xP^{\alpha}_t\varphi,h   \rangle-\int_0^t P^{\alpha}_{t-s}[\langle Ah+ D_xp^{\alpha}(x)h,D_xP^\alpha_s\varphi\rangle]  \,ds.
 \end{equation}
 where $h\in D(A)$ and $\varphi\in C^1_b(H)$.\bigskip

   \subsection{The range condition}
Let us consider  the Kolmogorov operator
 \begin{equation}
\label{e3.6}
\mathcal K u(t,x)=D_tu(t,x)+\langle F(t,x),D_xu(t,x)  \rangle,
\end{equation}
defined for all    $u\in \mathcal F C^1_{b,T},$   the space of all functions  $u$ defined in Section 1 with $Y=D(A)$.

Now the continuity equation \eqref{e1.2} can be  written as
 \begin{equation}
\label{e3.7}
\int_0^T\int_H\mathcal K u(t,x)\,\rho(t,x)\,\gamma(dx)dt=-\int_Hu(0,x)\rho_0(x)\,\gamma(dx),\quad u\in \mathcal F C^1_{b} .
\end{equation}

The following result has be proven in \cite{DaFlRoe14}.
\begin{Proposition}
\label{p3.1}
Assume that for $p\in[1,\infty)$  the following range condition is fulfilled
\begin{equation}
\label{e3.8}
\mathcal K (\mathcal FC^1_{b,T}) \;\mbox{\rm is dense in}\;L^{p}([0,T]; L^p(H,\gamma)).
\end{equation}
 Then if $\rho_1$ and $\rho_2$ are two solutions of \eqref{e3.7} in $L^{p'}([0,T]; L^{p'}(H,\gamma))$, with $p'=\tfrac{p}{p-1},\;pÕ=\tfrac{p}{p-1}$,
 we have $\rho_1 =\rho_2$.
\end{Proposition}

Let now consider    the approximating equation
 \begin{equation}
\label{e3.11}
\left\{\begin{array}{l}
D_tu_j(t,x)+ \langle F_j(t,x), D_xu_j(t,x)   \rangle=f(t,x),\\
\\
u_j(T,\cdot)=0,
\end{array}\right. 
\end{equation}
where  $(F_j)$ where defined in Hypothesis \ref{h2} and  $f\in \mathcal F C^1_{b,T}$. Problem \eqref{e3.11} has a unique classical solution given   by
\begin{equation}
\label{e3.12}
u_j(t,x)=-\int_t^{T}  f(s,\xi_j(s,t,x))ds.
\end{equation}
where $\xi_j$ is the solution to
\begin{equation}
\label{e3.13}
 \frac{d}{dt}\,\xi_j(t)=F_j(t,\xi_j(t)),\quad \xi_j(s)=x.
\end{equation}

Let us  consider a further approximation $P_\epsilon u_j(t,x)$ of $u(t,x)$, where $P_\epsilon$ is the transition semigroup defined in \eqref{e1} and $\epsilon\in(0,1]$.
 Applying  $P_\epsilon$ to both sides of equation \eqref{e3.11} we have 
  $$
D_t (P_\epsilon u_j)+   \langle F,D_xP_\epsilon u_j   \rangle=P_\epsilon f+ \langle F-F_j,D_xP_\epsilon u_j   \rangle+B_\epsilon(F_j,u_j),
$$
where $B_\epsilon(F_j,u_j)$ is the {\em  DiPerna--Lions} commutator defined  for $\epsilon\in(0,1]$  as
\begin{equation}
\label{e3.22c}
B_\epsilon(u,F)(t,x):=\langle D_xP_\epsilon u(t,x),F(t,x)\rangle -P_\epsilon(\langle D_xu(t,x),F(t,x)\rangle) ,\quad \forall\;u\in \mathcal F C^1_{b,T}\,,F\in\mathcal V\mathcal F C^1_{b,T}.
 \end{equation}
Now the range condition follows  provided
\begin{equation}
\label{er}
  \lim_{\epsilon\to 0}\lim_{j\to\infty}B_\epsilon(u_j,F_j)=0\quad \mbox{\rm in}\;u\in L^1([0,T], L^1(H,\gamma)).
  \end{equation}
   
  As shown in \cite{DaFlRoe14},   the basic tool  to show \eqref{er} is provided by an  estimate for the integral
  $$
 \int_0^T \int_H |B_\epsilon(u,F)|\,dt\,d\gamma,\quad \epsilon\in(0,1],\quad \forall\;u\in \mathcal F C^1_{b,T} \,,F\in\mathcal V\mathcal F C^1_{b,T} ,
  $$
in terms of $\|u\|_\infty$  independent of   $\epsilon$.   

\subsection{Main result}

To express the main result of this section we  need some definitions.

\begin{Definition}
We call $\mathcal{V}\left(  H,\gamma\right)  $ the space of all measurable
functions $\phi:H\rightarrow\mathbb{R}$ such that
\[
\left\Vert \phi\right\Vert _{\mathcal{V}\left(  H,\gamma\right)  }^{2}%
:=\sup_{\epsilon\in\left(  0,1\right)  }\int_{H}\phi\left(  x\right)  \left(
\frac{I-P_{\epsilon}}{\epsilon}\right)  \phi\left(  x\right)  \gamma\left(
dx\right)
\]
is finite and we endow $\mathcal{V}\left(  H,\gamma\right)  $ by the norm
$\left\Vert \phi\right\Vert _{\mathcal{V}\left(  H,\gamma\right)  }$.
Similarly we  call $\mathcal{V}\left(  H,H,\gamma\right)  $ the space of all
measurable vector fields $G:H\rightarrow\mathbb{R}$ such that
\[
\left\Vert G\right\Vert _{\mathcal{V}\left(  H,H,\gamma\right)  }^{2}%
:=\sup_{\epsilon\in\left(  0,1\right)  }\int_{H}\left\langle \left(
\frac{I-\mathbf{P}_{\epsilon}}{\epsilon}\right)  G\left(  x\right)  ,G\left(
x\right)  \right\rangle _{H}\gamma\left(  dx\right)
\]
is finite and we endow $\mathcal{V}\left(  H,H,\gamma\right)  $ by the norm
$\left\Vert G\right\Vert _{\mathcal{V}\left(  H,H,\gamma\right)  }$.
\end{Definition}
We note that in the symmetric case ($B=I$), $\mathcal{V}\left(  H,\gamma\right)  $ coincides
with $D((-\mathcal L)^{1/2})$.

\begin{Lemma}
\label{lemma after def new norm}The space $\mathcal{F}C_{b}^{2}(H)$ is contained
in $\mathcal{V}\left(  H,\gamma\right)  $. Similar result holds for every
vector field $G$ of the form $G=\sum_{h=1}^{n}G_{h}e_{h}$, with $G_{h}%
\in\mathcal{F}C_{b}^{2}$ for all $h=1,...,n$.
\end{Lemma}

\begin{proof}
We have%
\[
\left(  I-P_{\epsilon}\right)  \phi\left(  x\right)  =\int_{0}^{\epsilon}%
P_{s}\mathcal{L}\phi\left(  x\right)  ds
\]
where $\mathcal{L}$ is the infinitesimal generator of $P_{t}$. One can check
that $\mathcal{L}\phi$ is a bounded continuous function; in particular this is
true for the term $\left\langle p\left(  x\right)  ,D_{x}\phi\left(  x\right)
\right\rangle $ because the argument of $\phi$ is in the space of continuous
functions. Hence $\left(  \frac{I-P_{\epsilon}}{\epsilon}\right)  \phi$ is
also bounded and thus $\phi\in\mathcal{V}\left(  H,\gamma\right)  $.
\qed
\end{proof}

Finally, we have our main estimate. Given  $\theta_0\in(\tfrac14,\tfrac12)$ and    $\theta   \in\left(  \theta_{0},\frac{1}{2}\right)$, we define
\begin{align*}
\left\Vert F\right\Vert _{p,q,\gamma,T}  & :=\left\Vert \left(  -A\right)
^{\theta_{0}}F\right\Vert _{L^{\frac{p}{p-1}}\left(  0,T;\mathcal{V}\left(
H,H,\gamma\right)  \right)  }\\
& +\left\Vert \left(  -A\right)  ^{1/2+\theta}F\right\Vert _{L^{\frac{p}{p-1}%
}\left(  0,T;L^{q}\left(  H,\gamma\right)  \right)  }+\left\Vert
\operatorname{div}F\right\Vert _{L^{\frac{p}{p-1}}\left(  0,T;L^{\frac{p}%
{p-1}}\left(  H,\gamma\right)  \right)  }.
\end{align*}

 \begin{Theorem}
\label{tmain}
For every $p,q$ satisfying
$$
p   \in(2,\infty), \quad \frac{1}{p}+\frac{1}{q}   <1, 
$$
for every   vector field $F:\left[  0,T\right]  \times H\rightarrow
D\left(  \left(  -A\right)  ^{1/2+\theta}\right)  $ such that $\left\Vert
F\right\Vert _{p,q,\gamma,T}$ is finite, there is  at most one  solution of the continuity equation in  $L^{q'}([0,T]; L^{p'}(H,\gamma))$, with $p'=\tfrac{p}{p-1},\;q'=\tfrac{q}{q-1}$.
\end{Theorem}
\begin{proof}
The conclusion of the theorem follows from the rank condition proved in Theorem \ref{trank} below, and Proposition \ref{p3.1}.
\qed
\end{proof}

\begin{Corollary}
\label{c3.11primo}
If $B$ in \eqref{e3.1} is the identity, then  under the conditions of Theorem   \ref{tmain} there exists a unique solution of the continuity equation in $L^{q'}([0,T];L^{q'}(H,\gamma))$
\end{Corollary}
\begin{proof}
The existence follows by Theorem \ref{t2.1n} and Remark \ref{r2.9}.
\qed
\end{proof}
 
\begin{Remark}
 \label{r3.12primo}
 \em As already mentioned in Remark \ref{r1.5primo}, so far we cannot prove whether Hypothesis \ref{h2}(ii) holds for $\gamma$ as in Example
 \ref{ex1.3}(iii), if $B$ in \eqref{e3.1bis}, \eqref{e3.1} is not the identity operator. In this case it was proved in \cite{BoDaRoe96}, \cite{DaDe04}  that $\gamma$ has a density $f$ with respect to $\gamma_0:=N(0,\tfrac12\,(-A)^{-1})$
 such that $\sqrt f\in W^{1,2}(H,\gamma_0)$, i.e. the Sobolev space of order  $1$ in $L^2(H,\gamma_0)$. To verify Hypothesis \ref{h2}(ii) it would be enough to show that $x\mapsto f(x,y), (x,y)\in H_N\oplus E_N$, is continuous and strictly positive on $H_N$, for all $N\in\N$ and $\nu_N$--a.e. $y\in E_N$, where $A,\,H_N,\,E_N$ and $\nu_N$ are as in Example \ref{ex2.10}. However, so far  we did not succeed to prove this. If this could be shown, Corollary \ref{c3.11primo} would hold for any
$B$ in \eqref{e3.1bis}, \eqref{e3.1}.
 \end{Remark}

 \subsection{Estimating the   commutator}
 We first express the   DiPerna--Lions  commutator
 $ B_\epsilon(u,F) $ using the identity \eqref{e3.9aa}. It is convenient to introduce the approximating commutator
 \begin{equation}
\label{e3.22cc}
B^\alpha_\epsilon(u,F)(t,x):=D_xP^\alpha_\epsilon u(t,x)\cdot F(t,x)-P^\alpha_\epsilon(D_xu(t,x)\cdot F(t,x)),\quad \forall\;u\in \mathcal F C^1_{b,T}(H)\,,F\in\mathcal V\mathcal F C^1_{b,T}(H)
 \end{equation}
 for any $\alpha\in(0,1]$.
 
 \begin{Lemma}
 \label{l3.10}
 Assume that 
 $ F=\sum_{h=1}^n F^h e_h, $ with  $F^h \in \mathcal V\mathcal F C^1_{b,T}(D(A))$,  $h=1,...,n$. Then we have
 \begin{equation}
\label{e3.29k}
\begin{array}{l}
\ds  B^\alpha_\epsilon(u,F)=\frac1\epsilon\,\E \left [u(t,X_\alpha(\epsilon,x))(F(t,x)-F(t,X_\alpha(\epsilon,x))\cdot \int_0^\epsilon (D_xX_\alpha(\eta,x))^* \pi_n(B^{-1})^*dW(\eta)\right]\\
\\
\ds+ \int_0^\epsilon P^\alpha_{\epsilon-\eta}\Bigg\{\frac1\eta \E\Big[u(t,X_\alpha(\eta,x))\,\\
\\
\ds\hspace{20mm} \times \Big< F(t,X_\alpha(\eta,x)),\int_0^\eta  (A+Dp_\alpha(x)))(D_xX_\alpha(\lambda,x))^*\pi_n(B^{-1})^*\,dW(\lambda) \Big>\Big]\Bigg\}d\eta\\
\\
\hspace{20mm}+P^\alpha_\epsilon(u\,\mbox{\rm div}\,F),
\end{array} 
\end{equation}
where  $\pi_n$ is the orthogonal projector on $(e_1,...,e_n)$.
  \end{Lemma}
  \begin{proof}
 Taking  into account \eqref{e3.9aa}, we write
  \begin{equation}
\label{e3.29c}
\begin{array}{l}
\ds P^\alpha_\epsilon(Du\cdot F)=\ds\sum_{h=1}^nP^\alpha_\epsilon(D_hu\,F^h)=
\sum_{h=1}^n P^\alpha_\epsilon(D_h(uF^h))- P^\alpha_\epsilon(u\,\mbox{\rm div}\,F)\\
\\
\ds=\sum_{h=1}^n D_hP^\alpha_\epsilon(uF^h)- \sum_{h=1}^n\int_0^\epsilon P^\alpha_{\epsilon-\eta}\left[D_x P^\alpha_\eta(uF_h)\cdot(Ae_h+Dp_\alpha e_h)\right]d\eta-
P^\alpha_\epsilon(u\,\mbox{\rm div}\,F).
\end{array}
\end{equation}  
Therefore
 \begin{equation}
\label{e3.29d}
\begin{array}{l}
\ds  B^\alpha_\epsilon(u,F)=\sum_{h=1}^n[D_hP^\alpha_\epsilon(u)F_h-D_h(P^\alpha_\epsilon(uF_h)]\\
\\
\ds+\sum_{h=1}^n\int_0^\epsilon P^\alpha_{\epsilon-\eta}[D_xP^\alpha_\eta(uF_h)\cdot(Ae_h+D_xp_\alpha(x)e_h)]d\eta+P^\alpha_\epsilon(u\,\mbox{\rm div}\,F)\\
\\
=:I_1+I_2+I_3.
\end{array} 
\end{equation}
  Let us write $I_1$ and $I_2$  in a more compact way. Recalling the Bismut-Elworthy-Li formula \eqref{e4.10} we have
 \begin{equation}
\label{e3.29e}
\begin{array}{l}
\ds  I_1=\frac1\epsilon\,\sum_{h=1}^n\E \left [u(t,X_\alpha(\epsilon,x))(F_h(t,x)-F_h(t,X_\alpha(\epsilon,x))\int_0^\epsilon D_xX_\alpha(\eta,x)e_h)\cdot \pi_n(B^{-1})^*dW(\eta)\right]\\
\\
\ds=\frac1\epsilon\,\E \left [u(t,X_\alpha(\epsilon,x)(F(t,x)-F(t,X_\alpha(\epsilon,x))\cdot \int_0^\epsilon (D_xX_\alpha(\eta,x))^* \pi_n(B^{-1})^*dW(\eta)\right]
\end{array} 
\end{equation}
(the last integral is well defined because obviously $\pi_n(B^{-1})^*(X_x(\eta,x))^*$ is Hilbert--Schmidt.)  As for $I_2$ we have, using again \eqref{e4.10}
 \begin{equation}
\label{e3.29l}
\begin{array}{l}
\ds  I_2= \sum_{h=1}^n\int_0^\epsilon P^\alpha_{\epsilon-\eta}[D_xP^\alpha_\eta(uF_h)\cdot(Ae_h+D_xp_\alpha(x)e_h)]d\eta\\
\\
\ds=\sum_{h=1}^n\int_0^\epsilon P^\alpha_{\epsilon-\eta}\Bigg\{\frac1\eta \E\Big[u(t,X_\alpha(\eta,x))\,F_h(t,X_\alpha(\eta,x)))\\
\\
\ds\hspace{20mm} \times\int_0^\eta\langle B^{-1}D_xX_\alpha(\lambda,x)(A\pi_ne_h+D_xp_\alpha \pi_ne_h),dW(\lambda)\rangle\Big]\Bigg\}d\eta\\
\\
\ds= \int_0^\epsilon P^\alpha_{\epsilon-\eta}\Bigg\{\frac1\eta \E\Big[u(t,X_\alpha(\eta,x))\,  F(t,X_\alpha(\eta,x)))\\
\\
\ds\hspace{20mm} \cdot\int_0^\eta   (A+D_xp_\alpha (x))( (D_xX_\alpha(\eta,x))^*\,\pi_n(B^{-1})^*dW(\lambda) \Big]\Bigg\}d\eta.
\end{array} 
\end{equation}
So,  \eqref{e3.29k} follows.
\qed
\end{proof}
The following corollary is a consequence of Lemma \ref{l3.10} taking into account the invariance of $\gamma_\alpha$.
\begin{Corollary}
 \label{c3.10}
 Assume that 
 $ F=\sum_{h=1}^n F^h e_h, $ with  $F^h \in \mathcal F C^1_b(D(A))$,  $h=1,...,n$. Then we have, 
 \begin{equation}
\label{e3.29h}
\begin{array}{l}
\ds  \int_H|B^\alpha_\epsilon(u,F)|d\gamma_\alpha\\
\\
\ds\le\frac1\epsilon\,\int_H\E \left |u(t,X_\alpha(\epsilon,x))(F(t,x)-F(t,X_\alpha(\epsilon,x))\cdot \int_0^\epsilon (D_xX_\alpha(\eta,x))^* \pi_n(B^{-1})^*dW(\eta)\right|\,d\gamma_\alpha\\
\\
\ds+\int_H \int_0^\epsilon  \frac1\eta \E\Big|u(X_\alpha(\eta,x))\,  F(X(\eta,x)))\\
\\
\ds\hspace{20mm}\cdot \int_0^\eta (A+D_xp_\alpha(x)))(D_xX_\alpha(\eta,x))^*\pi_n(B^{-1})^*dW(\lambda)\Big|d\eta\,d\gamma_\alpha\\
\\
\ds+\int_H|u\,\mbox{\rm div}\,F|\,d\gamma_\alpha=:J_1+J_2+J_3.
\end{array} 
\end{equation}
 \end{Corollary}
 To estimate $\int_H|B_\epsilon(u,F)|\,d\gamma$ we need some preliminary results.
 \begin{Proposition}
\label{prop prima stima}For every $p\in(2,\infty]$ there is a constant
$C_{p}>0$, independent of $\alpha$ and $\epsilon$, such that
\begin{align*}
&  \frac{1}{\epsilon}\int_{H}\mathbb{E}\left[  \left\vert u\left(  t,x\right)
\left\langle F\left(  t,x\right)  -F\left(  t,X_{\alpha}\left(  \epsilon
,x\right)  \right)  ,\int_{0}^{\epsilon}\left(  D_{x}X_{\alpha}\left(
\eta,x\right)  \right)  ^{\ast}\left(  B^{-1}\right)  ^{\ast}dW\left(
\eta\right)  \right\rangle \right\vert \right]  \gamma_{\alpha}\left(
dx\right)  \\
&  \leq C_{A,B,p}\left(  \int_{H}\left\langle \left(  \frac{I-\mathbf{P}%
_{\epsilon}^{\alpha}}{\epsilon}\right)  \left(  -A\right)  ^{\theta_{0}%
}F\left(  t,x\right)  ,\left(  -A\right)  ^{\theta_{0}}F\left(  t,x\right)
\right\rangle _{H}\gamma_{\alpha}\left(  dx\right)  \right)  ^{1/2}\left(
\int_{H}\left\vert u\left(  t,x\right)  \right\vert ^{p}\gamma_{\alpha}\left(
dx\right)  \right)  ^{1/p}%
\end{align*}
where $C_{A,B,p}=C_{p}\left\Vert \left(  -A\right)  ^{-\theta_{0}}\right\Vert
_{L_{2}\left(  H\right)  }\left\Vert B^{-1}\right\Vert _{\mathcal{L}\left(
H\right)  }$ for some constant $C_{p}>0$.
\end{Proposition}

\begin{proof}
Call $I$ the integral we have to estimate. To shorten the notations, call
$I^{\prime}$ the stochastic integral%
\[
I^{\prime}:=\int_{0}^{\epsilon}\left(  -A\right)  ^{-\theta_{0}}\left(
D_{x}X_{\alpha}\left(  \eta,x\right)  \right)  ^{\ast}\left(  B^{-1}\right)
^{\ast}dW\left(  \eta\right)  .
\]
We have%
\begin{align*}
I &  =\frac{1}{\epsilon}\int_{H}\mathbb{E}\left[  u\left(  t,x\right)
\left\langle \left(  -A\right)  ^{\theta_{0}}F\left(  t,x\right)  -\left(
-A\right)  ^{1/2}F\left(  t,X_{\alpha}\left(  \epsilon,x\right)  \right)
,I^{\prime}\right\rangle \right]  \gamma_{\alpha}\left(  dx\right)  \\
&  \leq\frac{1}{\epsilon}\left(  \int_{H}\mathbb{E}\left[  \left\Vert \left(
-A\right)  ^{\theta_{0}}F\left(  t,x\right)  -\left(  -A\right)  ^{\theta_{0}%
}F\left(  t,X_{\alpha}\left(  \epsilon,x\right)  \right)  \right\Vert _{H}%
^{2}\right]  \gamma_{\alpha}\left(  dx\right)  \right)  ^{1/2}\\
&  \cdot\left(  \int_{H}\left\vert u\left(  t,x\right)  \right\vert ^{p}%
\gamma_{\alpha}\left(  dx\right)  \right)  ^{1/p}\left(  \int_{H}%
\mathbb{E}\left[  \left\Vert I^{\prime}\right\Vert _{H}^{r\left(  p\right)
}\right]  \gamma_{\alpha}\left(  dx\right)  \right)  ^{1/r\left(  p\right)  }%
\end{align*}
with $\frac{1}{p}+\frac{1}{2}+\frac{1}{r\left(  p\right)  }=1$ namely
$r\left(  p\right)  =\frac{p-2}{2p}$ and in particular with the condition
\[
p\in(2,\infty].
\]
By the Burkholder-Davies--Gundy inequality,
\begin{align*}
&  \mathbb{E}\left[  \left| I^{\prime}\right| _{H}^{r\left(  p\right)
}\right]  \leq C_{p}\mathbb{E}\left[  \left(  \int_{0}^{\epsilon}\left\Vert
\left(  -A\right)  ^{-\theta_{0}}\left(  D_{x}X_{\alpha}\left(  \eta,x\right)
\right)  ^{\ast}\left(  B^{-1}\right)  ^{\ast}\right\Vert _{L_{2}\left(
H\right)  }^{2}d\eta\right)  ^{r\left(  p\right)  /2}\right]  \\
&  \leq C_{p}\left\Vert \left(  -A\right)  ^{-\theta_{0}}\right\Vert
_{L_{2}\left(  H\right)  }^{r\left(  p\right)  }\left\Vert B^{-1}\right\Vert
_{\mathcal{L}\left(  H\right)  }^{r\left(  p\right)  }\mathbb{E}\left[
\left(  \int_{0}^{\epsilon}\left\Vert D_{x}X_{\alpha}\left(  \eta,x\right)
\right\Vert _{\mathcal{L}\left(  H\right)  }^{2}d\eta\right)  ^{r\left(
p\right)  /2}\right]  \\
&  \leq C_{p}\left\Vert \left(  -A\right)  ^{-\theta_{0}}\right\Vert
_{L_{2}\left(  H\right)  }^{r\left(  p\right)  }\left\Vert B^{-1}\right\Vert
_{\mathcal{L}\left(  H\right)  }^{r\left(  p\right)  }\left(  \sqrt{\epsilon
}\right)  ^{r\left(  p\right)  }%
\end{align*}
because, by dissipativity of the reaction diffusion system,
\[
\left\Vert D_{x}X_{\alpha}\left(  \eta,x\right)  \right\Vert _{\mathcal{L}%
\left(  H\right)  }\leq1.
\]
Therefore%
\[
I\leq\frac{C}{\sqrt{\epsilon}}\left(  \int_{H}\mathbb{E}\left[  \left|
\left(  -A\right)  ^{\theta_{0}}F\left(  t,x\right)  -\left(  -A\right)
^{\theta_{0}}F\left(  t,X_{\alpha}\left(  \epsilon,x\right)  \right)
\right |_{H}^{2}\right]  \gamma_{\alpha}\left(  dx\right)  \right)
^{1/2}\left(  \int_{H}\left\vert u\left(  t,x\right)  \right\vert ^{p}%
\gamma_{\alpha}\left(  dx\right)  \right)  ^{1/p}%
\]
where $C=C_{p}^{1/r\left(  p\right)  }\left\Vert \left(  -A\right)
^{-\theta_{0}}\right\Vert _{L_{2}\left(  H\right)  }\left\Vert B^{-1}%
\right\Vert _{\mathcal{L}\left(  H\right)  }$. Finally, writing $G\left(
t,x\right)  =\left(  -A\right)  ^{\theta_{0}}F\left(  t,x\right)  $,%
\begin{align*}
&  \int_{H}\mathbb{E}\left[  \left| \left(  -A\right)  ^{1/2}F\left(
t,x\right)  -\left(  -A\right)  ^{1/2}F\left(  t,X_{\alpha}\left(
\epsilon,x\right)  \right)  \right|_{H}^{2}\right]  \gamma_{\alpha
}\left(  dx\right)  \\
&  =\int_{H}\left(  \left| G\left(  t,x\right)  \right| _{H}%
^{2}-2\mathbb{E}\left[  \left\langle G\left(  t,x\right)  ,G\left(
t,X_{\alpha}\left(  \epsilon,x\right)  \right)  \right\rangle _{H}\right]
+\mathbb{E}\left[  \left| G\left(  t,X_{\alpha}\left(  \epsilon,x\right)
\right)  \right| _{H}^{2}\right]  \right)  \gamma_{\alpha}\left(
dx\right)  .
\end{align*}
Now%
\[
\mathbb{E}\left[  \left\langle G\left(  t,x\right)  ,G\left(  t,X_{\alpha
}\left(  \epsilon,x\right)  \right)  \right\rangle _{H}\right]  =\left\langle
G\left(  t,x\right)  ,\mathbb{E}\left[  G\left(  t,X_{\alpha}\left(
\epsilon,x\right)  \right)  \right]  \right\rangle _{H}=\left\langle G\left(
t,x\right)  ,\left(  \mathbf{P}_{\epsilon}^{\alpha}G\left(  t,\cdot\right)
\right)  \left(  x\right)  \right\rangle _{H}%
\]%
\begin{align*}
\int_{H}\mathbb{E}\left[  \left| G\left(  t,X_{\alpha}\left(
\epsilon,x\right)  \right)  \right|_{H}^{2}\right]  \gamma_{\alpha
}\left(  dx\right)   &  =\int_{H}\left(  P_{\epsilon}^{\alpha}\left|
G\left(  t,\cdot\right)  \right|_{H}^{2}\right)  \left(  x\right)
\gamma_{\alpha}\left(  dx\right)  \\
&  =\int_{H}\left| G\left(  t,x\right)  \right | _{H}^{2}\gamma
_{\alpha}\left(  dx\right)
\end{align*}
because $\gamma_{\alpha}$ is invariant for $P_{\epsilon}^{\alpha}$, hence%
\begin{align*}
&  \int_{H}\mathbb{E}\left[  \left|\left(  -A\right)  ^{1/2}F\left(
t,x\right)  -\left(  -A\right)  ^{1/2}F\left(  t,X_{\alpha}\left(
\epsilon,x\right)  \right)  \right|_{H}^{2}\right]  \gamma_{\alpha
}\left(  dx\right)  \\
&  =2\int_{H}\left(  \left|G\left(  t,x\right)  \right|_{H}%
^{2}-\left\langle G\left(  t,x\right)  ,\left(  \mathbf{P}_{\epsilon}^{\alpha
}G\left(  t,\cdot\right)  \right)  \left(  x\right)  \right\rangle
_{H}\right)  \gamma_{\alpha}\left(  dx\right)  \\
&  =2\int_{H}\left\langle G\left(  t,x\right)  ,G\left(  t,x\right)  -\left(
\mathbf{P}_{\epsilon}^{\alpha}G\left(  t,\cdot\right)  \right)  \left(
x\right)  \right\rangle _{H}\gamma_{\alpha}\left(  dx\right)  .
\end{align*}
Collecting these facts, we have proved the proposition.
\qed
\end{proof}

 \begin{Proposition}
\label{prop B alpha}Under the assumptions of Theorem \ref{tmain} 
there exist constants $C_{A,B,p}$ (given by Proposition \ref{prop prima stima}%
) and $C_{A,B,p,q,\theta}$, both independent of $\alpha$ and $\epsilon$, such
that
\begin{align*}
&  \int_{H}\left\vert B_{\epsilon}^{\alpha}\left(  u,F\right)  \left(
t,x\right)  \right\vert \gamma_{\alpha}\left(  dx\right)  \\
&  \leq C_{A,B,p}\left\Vert u\left(  t,\cdot\right)  \right\Vert
_{L^{p}\left(  H,\gamma_{\alpha}\right)  }\left(  \int_{H}\left\langle \left(
\frac{I-\mathbf{P}_{\epsilon}^{\alpha}}{\epsilon}\right)  \left(  -A\right)
^{\theta_{0}}F\left(  t,x\right)  ,\left(  -A\right)  ^{\theta_{0}}F\left(
t,x\right)  \right\rangle _{H}\gamma_{\alpha}\left(  dx\right)  \right)
^{1/2}\\
&  +C_{A,B,p,q,\theta}\left\Vert u\left(  t,\cdot\right)  \right\Vert
_{L^{p}\left(  H,\gamma_{\alpha}\right)  }\left\Vert \left(  -A\right)
^{1/2+\theta}F\left(  t,\cdot\right)  \right\Vert _{L^{q}\left(
H,\gamma_{\alpha}\right)  }\\
&  +\left\Vert u\left(  t,\cdot\right)  \right\Vert _{L^{p}\left(
H,\gamma_{\alpha}\right)  }\left\Vert \operatorname{div}F\left(
t,\cdot\right)  \right\Vert _{L^{\frac{p}{p-1}}\left(  H,\gamma_{\alpha
}\right)  }%
\end{align*}
for all functions $u\in\mathcal{F}C_{b,T}^{1}(H)$ and vector field $F$ of the
form $F=\sum_{h=1}^{n}F_{h}e_{h}$, with $F_{h}\in\mathcal{F}C_{b,T}^{2}(H)$ for
all $h=1,...,n$.
\end{Proposition}

\begin{proof}
\textbf{Step 1}. We know%
\[
\int_{H}\left\vert B_{\epsilon}^{\alpha}\left(  u,F\right)  \left(
t,x\right)  \right\vert \gamma_{\alpha}\left(  dx\right)  \leq J_{1}%
+J_{2}+J_{3}%
\]
where%
\[
J_{1}=\frac{1}{\epsilon}\int_{H}\mathbb{E}\left[  \left\vert u\left(
t,x\right)  \left\langle F\left(  t,x\right)  -F\left(  t,X_{\alpha}\left(
\epsilon,x\right)  \right)  ,\int_{0}^{\epsilon}\left(  D_{x}X_{\alpha}\left(
\eta,x\right)  \right)  ^{\ast}\left(  B^{-1}\right)  ^{\ast}dW\left(
\eta\right)  \right\rangle \right\vert \right]  \gamma_{\alpha}\left(
dx\right)
\]%
\[
J_{2}=\int_{H}\int_{0}^{\epsilon}\frac{1}{\eta}\mathbb{E}\left[  \left\vert
u\left(  t,X_{\alpha}\left(  \eta,x\right)  \right)  \left\langle F\left(
t,X_{\alpha}\left(  \eta,x\right)  \right)  ,J_{2}^{\prime}\right\rangle
\right\vert \right]  d\eta d\gamma_{\alpha}\left(  x\right)
\]%
\[
J_{3}=\int_{H}u\left(  t,x\right)  \operatorname{div}F\left(  t,x\right)
\gamma_{\alpha}\left(  dx\right)  .
\]
where for shortness we wrote%
\[
J_{2}^{\prime}=\int_{0}^{\eta}\left(  A+D_{x}p_{\alpha}\left(  x\right)
\right)  ^{\ast}\left(  D_{x}X_{\alpha}\left(  \lambda,x\right)  \right)
^{\ast}\left(  B^{-1}\right)  ^{\ast}dW\left(  \lambda\right)  .
\]
The estimate for $J_{1}$ has been made above and the estimate for $J_{3}$ is
trivial. We need only to estimate $J_{2}$. Let $r>0$ be such that
\[
\frac{1}{p}+\frac{1}{q}+\frac{1}{r}=1.
\]
Then%
\begin{align*}
J_{2} &  \leq\int_{0}^{\epsilon}\frac{1}{\eta}d\eta\left(  \int_{H}%
\mathbb{E}\left[  \left\vert u\left(  t,X_{\alpha}\left(  \eta,x\right)
\right)  \right\vert ^{p}\right]  \gamma_{\alpha}\left(  dx\right)  \right)
^{1/p}\\
&  \cdot\left(  \int_{H}\mathbb{E}\left[  \left| \left(  -A\right)
^{1/2+\theta}F\left(  t,X_{\alpha}\left(  \eta,x\right)  \right)  \right|_{H}^{q}\right]  \gamma_{\alpha}\left(  dx\right)  \right)  ^{1/q}\left(
\int_{H}\mathbb{E}\left[  \left|J_{2}^{\prime}\right| _{H}%
^{r}\right]  \gamma_{\alpha}\left(  dx\right)  \right)  ^{1/r}%
\end{align*}%
\begin{align*}
&  \leq\int_{0}^{\epsilon}\frac{1}{\eta}d\eta\left(  \int_{H}\left(  P_{\eta
}^{\alpha}\left(  \left\vert u\left(  t,\cdot\right)  \right\vert ^{p}\right)
\right)  \left(  x\right)  \gamma_{\alpha}\left(  dx\right)  \right)  ^{1/p}\\
&  \cdot\left(  \int_{H}\left(  P_{\eta}^{\alpha}\left(  \left| \left(
-A\right)  ^{1/2+\theta}F\left(  t,\cdot\right)  \right| _{H}^{q}\right)
\right)  \left(  x\right)  \gamma_{\alpha}\left(  dx\right)  \right)  ^{1/q}\\
&  \cdot\left(  \int_{H}\mathbb{E}\left[  \left(  \int_{0}^{\eta}\left\Vert
\left(  -A\right)  ^{-1/2-\theta}\left(  A+D_{x}p_{\alpha}\left(  x\right)
\right)  ^{\ast}\left(  D_{x}X_{\alpha}\left(  \lambda,x\right)  \right)
^{\ast}\left(  B^{-1}\right)  ^{\ast}\right\Vert _{L_{2}\left(  H\right)
}^{2}d\lambda\right)  ^{r/2}\right]  \gamma_{\alpha}\left(  dx\right)
\right)  ^{1/r}%
\end{align*}
and using invariance of $\gamma_{\alpha}$ for $P_{\eta}^{\alpha}$ and the fact
that $B^{-1}$ is bounded,
\[
J_2 \leq\left\Vert B^{-1}\right\Vert _{\mathcal{L}\left(  H\right)  }C\left(
\epsilon,\theta,r\right)  \left\Vert u\left(  t,\cdot\right)  \right\Vert
_{L^{p}\left(  H,\gamma_{\alpha}\right)  }\left\Vert \left(  -A\right)
^{1/2+\theta}F\left(  t,\cdot\right)  \right\Vert _{L^{q}\left(
H,\gamma_{\alpha}\right)  }%
\]
where $C\left(  \epsilon,\theta,r\right)$ and $g(x)$ are given respectively by:
\begin{align*}
&  \int_{0}^{\epsilon}\frac{1}{\eta}d\eta\left(  \int_{H}\mathbb{E}\left[
\left(  \int_{0}^{\eta}\left\Vert \left(  -A\right)  ^{-1/2-\theta}\left(
A+D_{x}p_{\alpha}\left(  x\right)  \right)  ^{\ast}\left(  D_{x}X_{\alpha
}\left(  \lambda,x\right)  \right)  ^{\ast}\right\Vert _{L_{2}\left(
H\right)  }^{2}d\lambda\right)  ^{r/2}\right]  \gamma_{\alpha}\left(
dx\right)  \right)  ^{1/r}\\
&  \leq\int_{0}^{\epsilon}\frac{1}{\eta}d\eta\left(  \int_{H}\mathbb{E}\left[
\left(  \int_{0}^{\eta}\left\Vert \left(  -A\right)  ^{1/2-\theta}\left(
D_{x}X_{\alpha}\left(  \lambda,x\right)  \right)  ^{\ast}\right\Vert
_{L_{2}\left(  H\right)  }^{2}d\lambda\right)  ^{r/2}\right]  g\left(
x\right)  \gamma_{\alpha}\left(  dx\right)  \right)  ^{1/r}%
\end{align*}%
\[
g\left(  x\right)  :=\left\Vert \left(  -A\right)  ^{-1/2-\theta}\left(
A+D_{x}p_{\alpha}\left(  x\right)  \right)  ^{\ast}\left(  -A\right)
^{-1/2+\theta}\right\Vert _{\mathcal{L}\left(  H\right)  }^{r}.
\]
It remains to estimate $C\left(  \epsilon,\theta,r,\theta\right)  $ (which a
priori may be infinite).

\textbf{Step 2}. From \cite[Corollary 2.3]{DaDe17}, we have, for $\delta\in\left(
0,1-\alpha\right)  $,%
\[
\int_{0}^{\eta}\left| \left(  -A\right)  ^{\left(  1-\alpha-\delta\right)
/2}D_{x}X_{\alpha}\left(  t,x\right)  h\right|_{H}^{2}dt\leq C\left(
T\right)  \Delta_{T}\left(  x\right)  \eta^{\delta}\left\Vert h\right\Vert
_{D\left(  \left(  -A\right)  ^{-\alpha/2}\right)  }^{2}%
\]
where%
\[
\Delta_{T}\left(  x\right)  =1+\sup_{t\in\left[  0,T\right]  }\left\Vert
D_{x}p_{\alpha}\left(  X_{\alpha}\left(  t,x\right)  \right)  \right\Vert
_{\infty}^{2}%
\]
(it is a random variable). In particular, choosing $\delta$ very small and
$\alpha=1-2\delta<1-\delta$, since the $H$ norm is bounded by any $D\left(
\left(  -A\right)  ^{\varepsilon}\right)  $-norm for $\varepsilon>0$, we get
\[
\int_{0}^{\eta}\left| D_{x}X_{\alpha}\left(  t,x\right)  h\right|
_{H}^{2}dt\leq C\left(  T\right)  \Delta_{T}\left(  x\right)  \eta^{\delta
}\left| h\right| _{D\left(  \left(  -A\right)  ^{-1/2+\delta}\right)
}^{2}.
\]
Hence, for $\delta=\theta-\theta_{0}$ (all constants denoted by $C,C\left(
T\right)  $ below, different from line to line, may depend on $T$ but not on
$\alpha$),
\begin{align*}
&  \int_{0}^{\eta}\left\Vert \left(  -A\right)  ^{1/2-\theta}\left(
D_{x}X_{\alpha}\left(  \lambda,x\right)  \right)  ^{\ast}\right\Vert
_{L_{2}\left(  H\right)  }^{2}d\lambda\\
&  =\int_{0}^{\eta}\left\Vert D_{x}X_{\alpha}\left(  \lambda,x\right)  \left(
-A\right)  ^{1/2-\theta}\right\Vert _{L_{2}\left(  H\right)  }^{2}d\lambda\\
&  =\sum_{k}\int_{0}^{\eta}\left|D_{x}X_{\alpha}\left(  \lambda,x\right)
\left(  -A\right)  ^{1/2-\theta}e_{k}\right|_{H}^{2}d\lambda\\
&  \leq C\left(  T\right)  \Delta_{T}\left(  x\right)  \eta^{2\left(
\theta-\theta_{0}\right)  }\sum_{k}\left|\left(  -A\right)  ^{1/2-\theta
}e_{k}\right|_{D\left(  \left(  -A\right)  ^{-1/2+\left(  \theta
-\theta_{0}\right)  }\right)  }^{2}\\
&  =C\left(  T\right)  \Delta_{T}\left(  x\right)  \eta^{\theta-\theta_{0}%
}\sum_{k}\left|\left(  -A\right)  ^{-\theta_{0}}e_{k}\right| _{H}%
^{2}\\
&  =C\left(  T\right)  \Delta_{T}\left(  x\right)  \eta^{\theta-\theta_{0}%
}\left\Vert \left(  -A\right)  ^{-\theta_{0}}\right\Vert _{L_{2}\left(
H\right)  }^{2}.
\end{align*}
Hence%
\begin{align*}
C\left(  \epsilon,\theta,r\right)   &  \leq\int_{0}^{\epsilon}\frac{1}{\eta
}d\eta\left(  \int_{H}\mathbb{E}\left[  \left(  C\left(  T\right)  \Delta
_{T}\left(  x\right)  \eta^{\theta-\theta_{0}}\left\Vert \left(  -A\right)
^{-\theta_{0}}\right\Vert _{L_{2}\left(  H\right)  }^{2}\right)
^{r/2}\right]  g\left(  x\right)  \gamma_{\alpha}\left(  dx\right)  \right)
^{1/r}\\
&  =C\left(  T\right)  ^{1/2}\left\Vert \left(  -A\right)  ^{-\theta_{0}%
}\right\Vert _{L_{2}\left(  H\right)  }\int_{0}^{\epsilon}\frac{\eta^{r\left(
\theta-\theta_{0}\right)  /2}}{\eta}d\eta\left(  \int_{H}\mathbb{E}\left[
\Delta_{T}\left(  x\right)  ^{r/2}\right]  g\left(  x\right)  \gamma_{\alpha
}\left(  dx\right)  \right)  ^{1/r}%
\end{align*}
It remains to bound
\begin{align*}
&  \int_{H}\mathbb{E}\left[  \Delta_{T}\left(  x\right)  ^{r/2}\right]
g\left(  x\right)  \gamma_{\alpha}\left(  dx\right)  \\
&  =\int_{H}\mathbb{E}\left[  \Delta_{T}\left(  x\right)  ^{r/2}\right]
\left\Vert \left(  -A\right)  ^{-1/2-\theta}\left(  A+D_{x}p_{\alpha}\left(
x\right)  \right)  ^{\ast}\left(  -A\right)  ^{-1/2+\theta}\right\Vert
_{\mathcal{L}\left(  H\right)  }^{r}\gamma_{\alpha}\left(  dx\right)  \\
&  \leq C\int_{H}\mathbb{E}\left[  \Delta_{T}\left(  x\right)  ^{r/2}\right]
\left(  1+\left\Vert \left(  -A\right)  ^{-1/2+\theta}D_{x}p_{\alpha}\left(
x\right)  \left(  -A\right)  ^{-1/2-\theta}\right\Vert _{\mathcal{L}\left(
H\right)  }^{r}\right)  \gamma_{\alpha}\left(  dx\right)  \\
&  \leq C\left(  \int_{H}\mathbb{E}\left[  \Delta_{T}\left(  x\right)
^{r}\right]  \gamma_{\alpha}\left(  dx\right)  \right)  ^{1/2}\cdot\\
&  \cdot\left(  \int_{H}\left(  1+\left\Vert \left(  -A\right)  ^{-1/2+\theta
}D_{x}p_{\alpha}\left(  x\right)  \left(  -A\right)  ^{-1/2-\theta}\right\Vert
_{\mathcal{L}\left(  H\right)  }^{2r}\right)  \gamma_{\alpha}\left(
dx\right)  \right)  ^{1/2}%
\end{align*}
renaming the constants. We have%
\[
\Delta_{T}\left(  x\right)  \leq1+C\sup_{t\in\left[  0,T\right]  }\left\Vert
X_{\alpha}\left(  t,x\right)  \right\Vert _{\infty}^{N-1}%
\]
and thus, by  \cite[Theorem 4.8 (iii)]{DaDe17},
\[
\mathbb{E}\left[  \Delta_{T}\left(  x\right)  ^{r}\right]  \leq C+C\left|
x\right| _{H}^{r\left(  N-1\right)  }%
\]
which implies%
\[
\int_{H}\mathbb{E}\left[  \Delta_{T}\left(  x\right)  ^{r}\right]
\gamma_{\alpha}\left(  dx\right)  \leq C.
\]
Finally, since%
\[
\left(  D_{x}p_{\alpha}\left(  x\right)  h\right)  \left(  \xi\right)
= p'(J_\alpha(x(\xi)))h(\xi)%
\]
we have%
\[
\left|D_{x}p_{\alpha}\left(  x\right)  h\right| _{H}\leq C\left\Vert
x\right\Vert _{\infty}^{N-1}\left\Vert h\right\Vert _{H}%
\]
namely%
\[
\left\Vert D_{x}p_{\alpha}\left(  x\right)  \right\Vert _{\mathcal{L}\left(
H\right)  }\leq C\left\Vert x\right\Vert _{\infty}^{N-1}%
\]
and therefore, being both $\left(  -A\right)  ^{-1/2+\theta}$ and $\left(
-A\right)  ^{-1/2-\theta}$ bounded in $H$ (recall that $\theta<\frac{1}{2}$),
\[
\left\Vert \left(  -A\right)  ^{-1/2+\theta}D_{x}p_{\alpha}\left(  x\right)
\left(  -A\right)  ^{-1/2-\theta}\right\Vert _{\mathcal{L}\left(  H\right)
}\leq\left\Vert D_{x}p_{\alpha}\left(  x\right)  \right\Vert _{\mathcal{L}%
\left(  H\right)  }\leq C\left\Vert x\right\Vert _{\infty}^{N-1}%
\]
which implies%
\[
\int_{H}\left(  1+\left\Vert \left(  -A\right)  ^{-1/2+\theta}D_{x}p_{\alpha
}\left(  x\right)  \left(  -A\right)  ^{-1/2-\theta}\right\Vert _{\mathcal{L}%
\left(  H\right)  }^{2r}\right)  \gamma_{\alpha}\left(  dx\right)  \leq C.
\]
\qed
\end{proof}
\begin{Corollary}
\label{corollary est}
Under the assumption of Theorem \ref{tmain} there exist constants $C_{A,B,p}$,
$C_{A,B,p,q,\theta}$, independent of $\epsilon$, such that
\begin{align*}
&  \int_{H}\left\vert B_{\epsilon}\left(  u,F\right)  \left(  t,x\right)
\right\vert \gamma\left(  dx\right)  
  \leq C_{A,B,p}\left\Vert u\left(  t,\cdot\right)  \right\Vert
_{L^{p}\left(  H,\gamma\right)  }\left\Vert \left(  -A\right)  ^{\theta_{0}%
}F\left(  t,\cdot\right)  \right\Vert _{\mathcal{V}\left(  H,H,\gamma\right)
}\\
&  +C_{A,B,p,q,\theta}\left\Vert u\left(  t,\cdot\right)  \right\Vert
_{L^{p}\left(  H,\gamma\right)  }\left\Vert \left(  -A\right)  ^{1/2+\theta
}F\left(  t,\cdot\right)  \right\Vert _{L^{q}\left(  H,\gamma\right)  }
  +\left\Vert u\left(  t,\cdot\right)  \right\Vert _{L^{p}\left(
H,\gamma\right)  }\left\Vert \operatorname{div}F\left(  t,\cdot\right)
\right\Vert _{L^{\frac{p}{p-1}}\left(  H,\gamma\right)  }%
\end{align*}
for all functions $u\in\mathcal{F}C_{b,T}^{1}$ and vector field $F$ of the
form $F=\sum_{h=1}^{n}F_{h}e_{h}$, with $F_{h}\in\mathcal{F}C_{b,T}^{2}$ for
all $h=1,...,n$.
\end{Corollary}

\begin{proof}
Let us consider term by term the main inequality of Proposition
\ref{prop B alpha}. Since $x\mapsto u\left(  t,\cdot\right)  $ is bounded
continuous function,
\[
\lim_{\alpha\rightarrow0}\left\Vert u\left(  t,\cdot\right)  \right\Vert
_{L^{p}\left(  H,\gamma_{\alpha}\right)  }^{p}=\lim_{\alpha\rightarrow0}%
\int_{H}\left\vert u\left(  t,x\right)  \right\vert ^{p}\gamma_{\alpha}\left(
dx\right)  =\int_{H}\left\vert u\left(  t,x\right)  \right\vert ^{p}%
\gamma\left(  dx\right)
\]
because $\gamma_{\alpha}$ converges weakly to $\gamma$. The same argument
applies to the terms $\left\Vert \left(  -A\right)  ^{1/2+\theta}F\left(
t,\cdot\right)  \right\Vert _{L^{q}\left(  H,\gamma_{\alpha}\right)  }$ and
$\left\Vert \operatorname{div}F\left(  t,\cdot\right)  \right\Vert
_{L^{\frac{p}{p-1}}\left(  H,\gamma_{\alpha}\right)  }$. 

We have to prove that
\[
\lim_{\alpha\rightarrow0}\int_{H}\left\vert B_{\epsilon}^{\alpha}\left(
u,F\right)  \left(  t,x\right)  \right\vert \gamma_{\alpha}\left(  dx\right)
=\int_{H}\left\vert B_{\epsilon}\left(  u,F\right)  \left(  t,x\right)
\right\vert \gamma\left(  dx\right)  .
\]
We have%
\[
\left\vert \int_{H}\left\vert B_{\epsilon}^{\alpha}\left(  u,F\right)  \left(
t,x\right)  \right\vert \gamma_{\alpha}\left(  dx\right)  -\int_{H}\left\vert
B_{\epsilon}\left(  u,F\right)  \left(  t,x\right)  \right\vert \gamma\left(
dx\right)  \right\vert \leq I_{1}+\left\vert I_{2}\right\vert
\]
where%
\begin{align*}
I_{1}  & =\int_{H}\left\vert \left\vert B_{\epsilon}^{\alpha}\left(
u,F\right)  \left(  t,x\right)  \right\vert -\left\vert B_{\epsilon}\left(
u,F\right)  \left(  t,x\right)  \right\vert \right\vert \gamma_{\alpha}\left(
dx\right)  \\
I_{2}  & =\int_{H}\left\vert B_{\epsilon}\left(  u,F\right)  \left(
t,x\right)  \right\vert \gamma_{\alpha}\left(  dx\right)  -\int_{H}\left\vert
B_{\epsilon}\left(  u,F\right)  \left(  t,x\right)  \right\vert \gamma\left(
dx\right)  .
\end{align*}
Recall that $\phi$ bounded continuous implies $x\mapsto\left(  P_{\epsilon
}^{\alpha}\phi\right)  \left(  x\right)  $ continuous and bounded by
$\left\Vert \phi\right\Vert _{\infty}$. One can prove that   when $\phi$ has also bounded continuous derivatives,
$x\mapsto\left(  D_{x}P_{\epsilon}^{\alpha}\phi\right)  \left(  x\right)  $ is
also continuous and uniformly bounded in $\alpha$. The same is true without
$\alpha$. Then $\left\vert B_{\epsilon}\left(  u,F\right)  \left(  t,x\right)
\right\vert $ is bounded continuous. It follows that $\left\vert
I_{2}\right\vert \rightarrow0$ as $\alpha\rightarrow0$, because $\gamma
_{\alpha}$ converges weakly to $\gamma$. Moreover, since the family $\left\{
\gamma_{\alpha}\right\}  $ is tight, given $\eta>0$ there is a compact set
$K_{\eta}\subset H$ such that $\gamma_{\alpha}\left(  K_{\eta}\right)
\geq1-\eta$ for all $\alpha$; and for what we have just said, outside
$K_{\eta}$ we may use the fact that $\left\vert B_{\epsilon}^{\alpha}\left(
u,F\right)  \left(  t,x\right)  \right\vert $ is uniformly bounded in $\alpha
$. Then we rewrite%
\[
I_{1}\leq\int_{K_{\eta}}\left\vert \left\vert B_{\epsilon}^{\alpha}\left(
u,F\right)  \left(  t,x\right)  \right\vert -\left\vert B_{\epsilon}\left(
u,F\right)  \left(  t,x\right)  \right\vert \right\vert \gamma_{\alpha}\left(
dx\right)  +C\eta.
\]
Recall that, when $\phi$ is bounded continuous, $P_{\epsilon}^{\alpha}\phi$
converges to $P_{\epsilon}\phi$ as $\alpha\rightarrow0$ uniformly on bounded
sets of $H$; and when $\phi$ has also bounded continuous derivatives, also
$D_{x}P_{\epsilon}^{\alpha}\phi$ converges to $D_{x}P_{\epsilon}\phi$ as
$\alpha\rightarrow0$, uniformly on bounded sets of $H$. Hence $\left\vert
\left\vert B_{\epsilon}^{\alpha}\left(  u,F\right)  \left(  t,x\right)
\right\vert -\left\vert B_{\epsilon}\left(  u,F\right)  \left(  t,x\right)
\right\vert \right\vert $ converges to zero uniformly on $K_{\eta}$. 

With the same argument, given $\phi\in\mathcal{F}C_{b}^{2}$, for every
$\epsilon$, we have
\[
\lim_{\alpha\rightarrow0}\int_{H}\phi\left(  x\right)  \left(  \frac
{I-P_{\epsilon}^{\alpha}}{\epsilon}\right)  \phi\left(  x\right)
\gamma_{\alpha}\left(  dx\right)  =\int_{H}\phi\left(  x\right)  \left(
\frac{I-P_{\epsilon}}{\epsilon}\right)  \phi\left(  x\right)  \gamma\left(
dx\right)  .
\]
Then, for every $\epsilon$,
\[
\lim_{\alpha\rightarrow0}\int_{H}\phi\left(  x\right)  \left(  \frac
{I-P_{\epsilon}^{\alpha}}{\epsilon}\right)  \phi\left(  x\right)
\gamma_{\alpha}\left(  dx\right)  \leq\left\Vert \phi\right\Vert
_{\mathcal{V}\left(  H,\gamma\right)  }^{2}.
\]
We apply this inequality in the vector case to $\left(  -A\right)
^{\theta_{0}}F\left(  t,\cdot\right)  $.
\qed
\end{proof}

Finally, we have our main estimate.

\begin{Theorem}
\label{trank}
Under the assumptions of Theorem \ref{tmain} there exist constants $C_{A,B,p}$,$C_{A,B,p,q,\theta}$ such that
\begin{align*}
  \int_{0}^{T}\int_{H}\left\vert B_{\epsilon}\left(  u,F\right)  \left(
t,x\right)  \right\vert \gamma\left(  dx\right)  dt
  \leq C_{A,B,p,\theta}\left\Vert u\right\Vert _{L^{p}\left(  0,T;L^{p}%
\left(  H,\gamma\right)  \right)  }\left\Vert F\right\Vert _{p,q,\gamma,T}%
\end{align*}
for all functions $u\in L^{p}\left(  0,T;L^{p}\left(  H,\gamma\right)
\right)  $ and vector fields $F:\left[  0,T\right]  \times H\rightarrow
D\left(  \left(  -A\right)  ^{1/2+\theta}\right)  $ such that $\left\Vert
F\right\Vert _{p,q,\gamma,T}$ is finite. Moreover, for such $\left(
u,F\right)  $,
\[
\lim_{\epsilon\rightarrow0}\int_{0}^{T}\int_{H}\left\vert B_{\epsilon}\left(
u,F\right)  \left(  t,x\right)  \right\vert \gamma\left(  dx\right)  dt=0.
\]
Under these conditions, the rank condition follows.
\end{Theorem}

\begin{proof}
The proof is similar to \cite{DaFlRoe14}.
\qed
\end{proof}

\newpage
\appendix
\noindent
\textbf{Appendix}

\renewcommand{\thesection}{A}  
\section{Deterministic Feynman--Kac formula and the solution of \eqref{e2.4a} for sufficiently regular $F$}
  
Consider the  equation
\begin{equation}
\label{eA1}
\left\{ \begin{array}{l}
\ds\frac{d}{dt}\,\xi(t) =\widetilde F(t,\xi(t)) , \\
\\
\xi(s)=x,\quad x\in \R^d,
\end{array} \right.
\end{equation}
with $\widetilde F$ regular, namely it belongs to  the class $\mathcal V \mathcal F C^1_b(H)$.
Let  $V\colon [0,T] \times \R^d\to \R$  be also regular.
We want to solve
\begin{equation}
 \label{eA2}
\left\{ \begin{array}{lll}
 v_{s}(s,x)&+&\displaystyle{\langle D_xv(s,x), \widetilde F(s,x)   \rangle +V(s,x)v(s,x)=0,  \quad 0\leq s< T, }\\
\\
v(T,x)&=&\varphi (x), \quad x \in H.
\end{array} \right.
\end{equation}
 The following result is well known, see e.g. \cite{Ma07}. We present, however, a proof for the reader's convenience.
\begin{Proposition}
\label{pA1}
Assume $\widetilde F\in C_b([0,T]\times \R^d;\R^d)$ such that $\widetilde F(t,\cdot)\in   C^1(\R^d,\R^d)$ for all $t\in [0,T]$ and let $V\in C([0,T]\times \R^d)$ such that  $V(t,\cdot)\in C^1(\R^d)$ for all $t\in[0,T]$ such that $D_xV:[0,T]\times \R^d\to  \R^d$ is continuous. Let $\varphi\in C^1(\R^d)$. Then the solution to \eqref{eA2} is given by
\begin{equation}
\label{eA3}
v(s,x)= \varphi (\xi(T,s,x))e^{\int_{s}^{T}V(u 
,\xi(u ,s,x))du }, \qquad (s,x)\in [0,T]\times \R^d,
\end{equation}
where for $s\le t$, $\xi(t,s,x)$ denotes the solution to   \eqref{eA1} at time $t$ when started at time $s$ at $x\in\R^d$. In particular,  $v(\cdot,x)\in C^1([0,T])$ for every $x\in\R^d$ and $D_tv\in C([0,T]\times \R^d)$.
\end{Proposition} 

\begin{proof}
We only present the main steps. We  shall check that $v$ defined by \eqref{eA3} is a solution to \eqref{eA2}.
 \medskip

For any decomposition $\{s=s_0<s_1<\cdots <s_n=T\} $  of $[s,T]$ we   write 
\begin{displaymath}
v(s,x)-\varphi (x)=-\sum_{k=1}^{n}[v(s_{k},x)-v(s_{k-1},x)],
\end{displaymath}
which is equivalent to,
\begin{equation}
\label{eA4}
\begin{array}{l}
\displaystyle{v(s,x)-\varphi (x)}=\ds{-\sum_{k=1}^{n}[v(s_{k},x)-v(s_{k},\xi(
s_{k},s_{k-1},x))]}\\
\\
\displaystyle{-\sum_{k=1}^{n}[v(s_{k},\xi(s_{k},s_{k-1},x))-v(s_{k-1},x)]=:J_1-J_2.
}
\end{array}
\end{equation}
Concerning $J_1$ we write thanks to Taylor's formula
\begin{equation}
\label{eA5}
\begin{array}{l}
\ds J_1\sim \sum_{k=1}^{n}\langle D_{x}v(s_{k},x),\xi(s_{k},s_{k-1},x)-x  \rangle\sim
\sum_{k=1}^{n}\langle D_{x}v(s_{k},x),\widetilde F(s_k,x) \rangle (s_{k}-s_{k-1})\\
\\
\ds\to \int_s^T \langle D_{x}v(r,x),\widetilde F(r,x) \rangle dr.
\end{array}
\end{equation}\bigskip

Concerning $J_2$ we write   \footnote{In the second line below we use that $\xi(T,s_k,\xi(s_{k},s_{k-1},x))=\xi(T,s_{k-1},x)$}
\begin{equation}
\label{eA6}
\begin{array}{l}
\ds J_2=\sum_{k=1}^nv(s_{k},\xi(s_{k},s_{k-1},x))-v(s_{k-1},x))\\
\\
\ds=\sum_{k=1}^n\varphi (\xi(T,s_k,\xi(s_{k},s_{k-1},x)))e^{\int_{s_k}^{T}V(u ,\xi(u ,s_k,\xi(s_{k},s_{k-1},x)))du
}\\
\\
\ds-\sum_{k=1}^n\varphi (\xi(T,s_{k-1},x))e^{\int_{s_{k-1}}^{T}V(u ,\xi(u ,s_{k-1},x))du }\\
\\
\ds=\sum_{k=1}^n\varphi (\xi(T,s_{k-1},x))\left[e^{\int_{s_k}^{T}V(u ,\xi(u ,s_{k-1},x))du}-e^{\int_{s_{k-1}}^{T}V(u ,\xi(u ,s_{k-1},x))du 
}\right]\\
\\
\ds=\sum_{k=1}^n v(s_{k-1},x))
\left( e^{-\int_{s_{k-1}}^{s_k}V(u ,\xi(u ,s_{k-1},x))du }-
1 \right)\\
\\
\ds\sim - \sum_{k=1}^nv(s_{k-1},x)V(s_{k-1},x)(s_{k}-s_{k-1})\to-\int_s^T v(r,x)V(r,x)dr.
\end{array}
\end{equation}

Replacing $J_1$  and $J_2$ given by \eqref{eA5}  and \eqref{eA6} respectively in \eqref{eA4}, yields
$$
v(s,x)=\varphi(x)+\int_s^T \langle D_{x}v(r,x),\widetilde F(r,x) \rangle dr+\int_s^T v(r,x)V(r,x)dr
$$
and the claim is proved.
\qed
\end{proof}

As a trivial consequence we obtain
\begin{Corollary}
\label{cA2}
Let $\Psi\in C^2(\R^d)$, $\Psi$ bounded  and strictly positive.
  Let $F\in C_b([0,T]\times \R^d;\R^d)$ such  that $F(t,\cdot)\in C^1(\R^d;\R^d)$ and  define
  $$
  D^*_xF(t,\cdot):=-\mbox{\rm div}\,F(t,\cdot)-\langle F(t,\cdot), D_x\Psi/\Psi\rangle_{\R^d}.
  $$
  Assume that $D^*_xF(t,\cdot)\in C^1(\R^d)$ for all $t\in[0,T]$, and $D^*_xF\in C([0,T]\times \R^d)$,
$D_xD^*_xF\in C([0,T]\times \R^d;\R^d)$. Then for every $\rho_0\in C^1(\R^d),\,\rho_0\ge 0,$
$$
\rho(t,x):= \rho_0 (\xi(T,T-t,x))e^{\int_{0}^{t}D_x^*F(T-u 
,\xi(T-u ,T-t,x))du }
$$
is a solution of \eqref{e2.4a}, where $\xi(\cdot,s,x)$ is the solution to \eqref{eA1} started at time $s$ at $x\in\R^d$, with $\widetilde F(t,x):=-F(T-t,x),\; (t,x)\in[0,T]\times \R^d$. Furthermore, $\rho(\cdot,x)\in C^1([0,T])$ for every $x\in\R^d$ and $D_t\rho\in C([0,T]\times \R^d).$

\end{Corollary}
\begin{proof}
Apply Proposition \ref{pA1} with $\widetilde F$ as in the assertion above,
$$
V(t,x)=D_x^* F(T-t,x),\quad (t,x)\in[0,T]\times\R^d
$$
and $\varphi:=\rho_0$.
\qed
\end{proof}

\renewcommand{\thesection}{B}
\section{A remark on the Burkholder--Davis--Gundy inequality}

Our aim in this section is to prove the following proposition.
\begin{Proposition}
\label{pB1}
Let  $p\ge 4$. Then for every $t\ge 0,$
\begin{equation}
\label{eB1}
\E\sup_{s\in [0,t]}\left|\int_0^t\Phi(s)dW(s)   \right|^p   \le
c_p\left[\E\left(\int_0^t\|\Phi(s)\|^2_{L_2^0}\,ds\right)^{p/2}      \right],
\end{equation}
where $c_p:=12^p\,p^p$.
\end{Proposition}
\begin{proof}
Set $$Z(t)=\int_{0}^{t}\Phi (s)dW(s), \quad
t\geq 0,$$ and apply It\^o's  formula to $f(Z(\cdot ))$
where $f(x)=|x|^{p}, \; x \in H$. Since
 \begin{displaymath}
f_{xx}(x)=p(p-2)|x |^{p-4}x \otimes x+p|x |^{p-2}I,\quad x \in H,
\end{displaymath}
we have
\begin{displaymath}
\|f_{xx}(x)\|\leq p(p-1)|x|^{p-2},
\end{displaymath}
therefore
\begin{displaymath}
|\mbox{\rm Tr}\;\Phi ^{*}(t)f_{xx}(Z(t))\Phi (t)Q| \leq  p(p-1) |Z(t)|^{p-2}\|\Phi (t)\|^{2}_{L_{2}^{0}}.
\end{displaymath}
By taking expectation  in the identity
$$
|Z(t)|^p=p \int_0^t |Z(s)|^{p-2}\langle Z(s),dZ(s)\rangle+\frac12\;\int_0^t\mbox{\rm Tr}\;[\Phi ^{*}(s)f_{xx}(Z(s))\Phi (s)Q]ds,
$$
 we obtain by the Burkholder--Davis--Gundy inequality for $p=1$
\begin{equation}
\label{eB2}
 \begin{array}{l}
 \ds\E\sup_{s\in[0,t]}|Z(s)|^{p}\le \frac{p(p-1)}2\;\E\left( \int_{0}^{t}|Z(s)|^{p-2}\|\Phi (s)\|^{2}_{L_{2}^{0}}\,ds  \right)\\
 \\
 \ds+3p\E\left[\left( \int_{0}^{t}\|\Phi (s)\|^{2}_{L_{2}^{0}}\,|Z(s)|^{2p-2}ds \right)^{1/2} \right]\\
\\
\ds\le \frac{p(p-1)}2\;\E\left(  \sup_{s \in [0,t]}|Z(s)|^{p-2} \int_{0}^{t}\|\Phi (s)\|^{2}_{L_{2}^{0}} ds  \right)\\
\\
\ds+3p\E\left[ \sup_{s \in [0,t]}|Z(s)|^{p-1} \left( \int_{0}^{t}\|\Phi (s)\|^{2}_{L_{2}^{0}} ds \right)^{1/2}   \right]\\
\\
\ds \le \frac{p(p-1)}2\;\left[\E\left(  \sup_{s \in [0,t]}\,|Z(s)|^p   \right)   \right]^{\frac{p-2}{p}} \;\left[\E\left(\int_{0}^{t}\|\Phi (s)\|^{2}_{L_{2}^{0}} ds    \right)^{\frac{p}{2} }  \right]^{\frac2p}\\
\\
\ds+3p\E\left[  \sup_{s \in [0,t]}\,|Z(s)|^p   \right]^{\frac{p-1}{p}}\;\left[\E\left(\int_{0}^{t}\|\Phi (s)\|^{2}_{L_{2}^{0}} ds    \right)^{\frac{p}{2} }  \right]^{\frac1p}\\
\\
:=J_1+J_2.
\end{array}
\end{equation}
 For $J_1$ we use Young's inequality with exponents $\tfrac{p}{p-2}$ and $\tfrac{p}{2}$ and find
 $$
 J_1\le \frac{1}4\,\E\left[  \sup_{s \in [0,t]}\,|Z(s)|^p   \right]+2^{p-1}p^p\,\E\left(\int_{0}^{t}\|\Phi (s)\|^{2}_{L_{2}^{0}} ds    \right)^{\frac{p}2}
 $$
 For $J_2$  we use Young's inequality with exponents $\tfrac{p}{p-1}$ and $p$ and find
 $$
 J_2\le \frac14\,E\left[  \sup_{s \in [0,t]}\,|Z(s)|^p   \right]+\frac12\,12^p\,p^p\, \E\left(\int_{0}^{t}\|\Phi (s)\|^{2}_{L_{2}^{0}} ds    \right)^{\frac{p}2}.
  $$
  Now  \eqref{eB1} with $c_p:=12^p\,p^p$ follows.
  \qed    
 \end{proof}

\renewcommand{\thesection}{C}
\section{Density of $\mathcal FC^1_b$ in Orlicz spaces}
Let $N:\R\to [0,\infty)$ be continuous and a Young function, i.e. convex, even and $N(0)=0$.

Consider the measure space $(H,\mathcal B(H),\gamma)$, where $H$ is as before a separable real Hilbert space with Borel $\sigma$--algebra $\mathcal B(H)$ and $\gamma$ a nonnegative finite measure on $(H,\mathcal B(H))$. We recall 
that the Orlicz space $L_N$ corresponding to $N$ is defined as
$$
L_N:=L_N(H, \gamma):=\{f:H\to\R:\, f \mbox{\it is $  \mathcal B(H)$--measurable and }\;\int_HN(af)d\gamma<\infty\; \mbox{\it for some}\; a>0  \}
$$
or equivalently
$$
L_N: =\{f:H\to\R:\, f\;\mbox{\it  is $\mathcal B(H)$--measurable and }\;\|f\|_{L_N}<\infty \},
$$
where
$$
\|f\|_{L_N}:=\inf\left\{\lambda>0:\,\int_HN(f/\lambda)\,d\gamma\le 1\right\}.
$$
$(L_N,\,\|\cdot\|_{L_N})$ is a Banach space (see e.g. \cite{RaRe02}).
\begin{Proposition}
\label{pC1}
$\mathcal F C^1_b$ is dense in $((L_N,\,\|\cdot\|_{L_N})$, where $\mathcal F C^1_b$ is defined as in Section 1. Furthermore, if $f\in L_N$, $f\ge 0$, then there exist nonnegative $f_n\in \mathcal F C^1_b,\;n\in\N$, such that
$$
\lim_{n\to\infty}\|f-f_n\|_{L_N}=0.
$$
Both assertions remain true, if $\mathcal F C^1_b$ is replaced by $\mathcal F C^1_0$
\end{Proposition}
\begin{proof}
We need the following lemma whose proof is straightforward, see e.g. \cite[Lemma 1.16]{Le}
\begin{Lemma}
\label{lC2}
Let $f_n\in L_N,\,n\in\N.$ Then the following assertions are equivalent:\medskip

\noindent(i) $\ds\lim_{n\to\infty} \|f_n\|_{L_N}=0$

\medskip

\noindent(ii)  For all $a\in(0,\infty)$
$$
\limsup_{n\to\infty}\int_HN(af_n)\,d\gamma\le 1
$$\medskip

\noindent(iii)  For all $a\in(0,\infty)$
$$
\lim_{n\to\infty}\int_HN(af_n)\,d\gamma=0.
$$

\end{Lemma}
{\bf Proof of Proposition \ref{pC1}}.

We shall use a monotone class argument. Define
$$
\begin{array}{l}
\mathcal M:=\Big\{ f:H\to\R:\,f\,\mbox{\it bounded, $\mathcal B(H)$--measurable such that} \\
\\
\ds  \hspace{10mm}\lim_{n\to\infty} \|f-f_n\|_{L_N}=0 ,\;\mbox{\it for some}\; f_n\in \mathcal FC^1_b,\,n\in\N\Big\}.
\end{array}
$$
Obviously, $\mathcal M$ is a linear space, $ \mathcal FC^1_b\subset  \mathcal M$ and  $ \mathcal FC^1_b$ is closed under multiplication and contains the constant function $1$. Furthermore, if $0\le u_n\in \mathcal M,$ $n\in \N,$ such that $u_n\uparrow u$ as $n\to \infty$ for some bounded $u:H\to [0,\infty)$, then for each $n\in\N$ there exists $ f_n\in \mathcal FC^1_b$ such that
\begin{equation}
\label{eC1}
\|u_n-f_n\|_{L_N}\le \frac1n.
 \end{equation}
 But since $N$ is continuous on $\R$, hence locally bounded, we have that for every $a\in (0,\infty)$, $N(a(u-u_n)),\,n\in\N,$ are uniformly bounded. Consequently, by Lebesgue's dominated convergence theorem and Lemma \ref{lC2}, we conclude that
 \begin{equation}
\label{eC2}
\lim_{n\to\infty}\|u-u_n\|_{L_N}=0.
 \end{equation}
 \eqref{eC1} and \eqref{eC2} imply that $u\in \mathcal M$, and therefore $\mathcal M$ is a monotone vector space and thus by the  monotone class theorem $\mathcal M$ is equal to the set of all bounded $\sigma(\mathcal F C^1_b)$--measurable functions on $H$. But $\sigma(\mathcal F C^1_b)=\mathcal B(H)$, since the weak and norm--Borel $\sigma$--algebra on a separable Banach space coincide. Hence  $\mathcal M$  is equal to all bounded $\mathcal B(H)$--measurable functions on $H$. Since by Lemma \ref{lC2} and the same arguments as above every $f$ in $L_N$ can be approximated in the norm $\|\cdot\|_{L_N}$ by bounded $\mathcal B(H)$--measurable functions, the first assertion of the proposition is proved.
 
  Now let $f\in L_N,\,f\ge 0.$    By the argument above we may assume that $f$ is bounded. Then by what we have just proved we can find
 $f_n\in \mathcal F C^1_b$ such that
 $$
 \lim_{n\to\infty}\|f-f_n\|_{L_N}=0.
 $$
 Since $|f-f_n^+|=|f^+-f_n^+|\le |f-f_n|$ for all $n\in\N$ and $N$ is even and increasing on $[0,\infty)$ (because $N$ is convex and $N(0)=0$), Lemma \ref{lC2} immediately implies that
 $$
 \lim_{n\to\infty}\|f-f_n^+\|_{L_N}=0.
 $$
 Fix $n\in\N$ and for $\epsilon>0$ take an increasing  function $\chi_\epsilon\in C^1(\R)$, $\chi_\epsilon(s)=s,\,\forall \,s\in[0,\infty)$ and  $\chi_\epsilon(s)=-\epsilon$ if  $s\in(-\infty,-2\epsilon)$. Then for each $n\in\N$
 $$
 \lim_{m\to\infty}\left\|f_n^+-\left(\chi_{\frac1m}(f_n)+\frac1m\right)\right\|_{\infty}=0.
 $$
So, again by Lemma \ref{lC2} and Lebesgue's dominated convergence theorem it follows that
$$
 \lim_{m\to\infty}\left\|f_n^+-\left(\chi_{\frac1m}(f_n)+\frac1m\right)\right\|_{L_N}=0.
 $$
 But obviously, $\chi_{\frac1m}(f_n)+\frac1m\in \mathcal F C^1_b$, $m\in\N$,  and each such function is nonnegative. Hence the second part of the assertion follows. The third part of the assertion then follows by similar arguments and multiplying by a  sequence of suitable localizing functions.
\qed
\end{proof}
\begin{Corollary}
\label{cC3}
Let $\rho\ge 0$, $\mathcal B(H)$--measurable  such that
$$
\int_H\rho\,\log\rho\,d\gamma<\infty.
$$
Then there exist nonnegative $\rho_n\in \mathcal F C^1_b,$ $n\in\N,$
such that
$$
 \lim_{n\to\infty}\rho_n=\rho\quad\mbox{\it in}\;L^1(H,\gamma)
$$
and
$$
\sup_{n\in\N}\int_H\rho_n\,\log\rho_n\,d\gamma<\infty.
$$
\end{Corollary}
\begin{proof}
Let $N(s):=(|s|+1)\,\ln(|s|+1)-|s|,\, s\in\R$.  Then it is easy to check that $N$ is a continuous Young function. Hence by Proposition \ref{pC1} we can find $\rho_n\in \mathcal F C^1_b$, $\rho_n\ge 0$, $n\in\N,$  such that
\begin{equation}
\label{eC3}
\lim_{n\to\infty}\|\rho-\rho_n\|_{L_N}=0.
 \end{equation}
 Since $L_N\subset L^1(H,\gamma)$ continuously (see \cite[Proposition 1.15]{Le}), the first assertion follows. Furthermore, we have for all $s\in(0,\infty)$
 $$
 s\ln s-s\le s\ln(s+1)\le (s+1)\ln(s+1)-s=N(s)
 $$
 and hence for $n\in\N$ by the convexity of $N$ and every $a\in(0,\infty)$
 $$
 \begin{array}{l}
 \ds \int_H\rho_n\,\ln \rho_n
\,d\gamma=\frac1a\int_H a\rho_n\,\ln(a \rho_n)\,\,d\gamma-\ln a\int_H\rho_n\,d\gamma\\
\\
\ds\le \frac1a\int_H N(a\rho_n)\,\,d\gamma+|1-\ln a|\int_H\rho_n\,d\gamma\\
\\
\ds\le \frac1{2a}\int_H N(2a(\rho_n-\rho))\,\,d\gamma+ \frac1{2a}\int_H N(2a\rho)\,\,d\gamma+|1-\ln a|\int_H\rho_n\,d\gamma.
\end{array} 
$$
Hence by the first part of the assertion, \eqref{eC3} and  Lemma \ref{lC2}, it follows that 
$$
\limsup_{n\to\infty}\int_H\rho_n\,\ln \rho_n\,d\gamma\le \frac1{2a}\int_H N(2a\rho)\,\,d\gamma+|1-\ln a|\int_H\rho\,d\gamma.
$$
But since $\rho\in L_N$ we can find $a>0$ such that the right hand side is finite. Hence   the second part of the assertion also follows.
\qed
\end{proof}

\section*{Acknowledgments}
G. Da Prato and F. Flandoli are partially supported by GNAMPA from INdAM.
M. R\"ockner is supported by  SFB 1283 through the DFG.
We also would like to thank an anonymous referee for his comments which led to an improvement of this paper.

\section*{Note added in Proof}
After this paper had been accepted for publication by JMPA in final form, we noticed that as a simple consequence of Proposition 6.4.1 in \cite{Boga10}, our Hypothesis \ref{h2}(ii) is in fact a consequence of our Hypothesis \ref{h1}, Lemma \ref{l2.2n} and \eqref{e1.8primo}. Hypothesis \ref{h2}(ii) can hence be dropped. In particular, our results therefore also apply to our Example \ref{ex1.3}(iii) and Remarks \ref{r1.5primo} and \ref{r3.12primo} can be dropped as well. We would like to thank Alexander Shaposhnikov for pointing out this particular result in the above reference to us.

\end{document}